\documentclass{article}

\usepackage[T1]{fontenc}     
\usepackage[utf8]{inputenc}
\usepackage[a4paper]{geometry}
\usepackage{amsmath}
\usepackage{mathtools}
\usepackage{graphicx}
\usepackage[english]{babel}
\usepackage{hyperref}
\usepackage{csquotes}

\usepackage{amssymb,amsfonts,amsthm}
\usepackage{dsfont}
\usepackage{xcolor}
\usepackage{todonotes}
\usepackage{appendix}
\usepackage{chemist}
\usepackage{caption}
\usepackage{subcaption}
\usepackage{tikz}
\usetikzlibrary{math}
\usetikzlibrary{decorations.pathreplacing}
\usepackage{placeins}
\usepackage{stmaryrd}

\numberwithin{equation}{section}

\usepackage[backend=biber, style=numeric, url=false]{biblatex}
\newtheorem{lemma}{Lemma}
\newtheorem{proposition}{Proposition}
\newtheorem{theorem}{Theorem}
\newtheorem*{theorem*}{Theorem}
\newtheorem{remark}{Remark}
\newtheorem{definition}{Definition}

\newcommand{\N}{\mathbb N}

\newcommand{\R}{\mathbb R}

\newcommand{\cA}{\ensuremath{\mathcal{A}}}

\newcommand{\cH}{\ensuremath{\mathcal{H}}}
\newcommand{\cI}{\ensuremath{\mathcal{I}}}

\newcommand{\cM}{\ensuremath{\mathcal{M}}}

\newcommand{\cV}{\ensuremath{\mathcal{V}}}

\newcommand{\bR}{\ensuremath{\mathbb{R}}}

\newcommand{\bbu}{\boldsymbol{u}}
\newcommand{\bbc}{\boldsymbol{c}}
\newcommand{\bbd}{\boldsymbol{d}}
\newcommand{\bbv}{\boldsymbol{v}}
\newcommand{\bbm}{\boldsymbol{m}}

\newcommand{\bbz}{\boldsymbol{z}}

\newcommand{\bbF}{\boldsymbol{F}}

\newcommand{\bbH}{\boldsymbol{H}}
\newcommand{\bbJ}{\boldsymbol{J}}
\newcommand{\bbA}{\boldsymbol{A}}
\newcommand{\bbM}{\boldsymbol{M}}

\newcommand{\bbI}{\boldsymbol{I}}

\newcommand{\bbmu}{\boldsymbol{\mu}}

\newcommand{\bbbeta}{\boldsymbol{\beta}}

\newcommand{\ovA}{\bar{A}}
\newcommand{\ovkappa}{\bar{\kappa}}

\newcommand{\ovc}{\bar{c}}
\newcommand{\ovX}{\overline{X}}

\renewcommand{\l}{\left}
\renewcommand{\r}{\right}

\newcommand{\blue}{\textcolor{blue}}

\title{Cross-diffusion systems coupled via a moving interface}
\author{Clément Cancès, Jean Cauvin-Vila, Claire Chainais-Hillairet, Virginie Ehrlacher}
\date{}

\begin{document}

\maketitle

\begin{abstract}
    We propose and study a one-dimensional model which consists of two cross-diffusion systems coupled via a moving interface. The motivation stems from the modelling of complex diffusion processes in the context of the vapor deposition of thin films. In our model, cross-diffusion of the various chemical species can be respectively modelled by a size-exclusion system for the solid phase and the Stefan-Maxwell system for the gaseous phase. The coupling between the two phases is modelled by linear phase transition laws of Butler-Volmer type, resulting in an interface evolution. The continuous properties of the model are investigated, in particular its entropy variational structure and stationary states. We introduce a two-point flux approximation finite volume scheme. The moving interface is addressed with a moving-mesh approach, where the mesh is locally deformed around the interface. The resulting discrete nonlinear system is shown to admit a solution that preserves the main properties of the continuous system, namely: mass conservation, nonnegativity, volume-filling constraints, decay of the free energy and asymptotics. In particular, the moving-mesh approach is compatible with the entropy structure of the continuous model. Numerical results illustrate these properties and the dynamics of the model. 
\end{abstract}

\section{Introduction}
\label{sec:intro}

We propose and study an extension of the mathematical model introduced in \cite{bakhta2018} to describe a physical vapor deposition process used in particular for the fabrication of semiconducting thin film layers in the photovoltaic industry. The process can be described as follows: a wafer is introduced in a hot chamber where chemical elements are injected under gaseous form. As the latter deposit on the substrate, a heterogeneous solid layer grows upon it. Because of the high temperature conditions, diffusion occurs in the bulk until the wafer is taken out and the system is frozen. There are two essential features in the problem: the evolution of the surface of the film and the diffusion of the various species due to high temperature conditions. In the series of works \cite{bakhta2018,bakhta2021,cauvin-vila2023}, the authors introduced and studied a one-dimensional moving-boundary cross-diffusion model where only the evolution of the solid layer was considered. The latter is composed of $n$ different chemical species and occupies a domain of the form $(0,X(t))$, where $X(t) >0$ denotes the thickness of the film at time $t>0$. For any $i \in \{1, \dots, n\}$, the flux of atoms of species $i$ absorbed at the surface of the solid film layer at time $t$ is denoted by $F_i(t)$. For all $t>0$ and $x\in (0, X(t))$, denoting by $c_{i}(t,x)$ the local volume fraction of species $i$ at position $x\in (0,X(t))$ and time $t$ and  setting $\bbF(t) = (F_i(t))_{1\leq i \leq n}$ and  $\bbc(t,x):= (c_{i}(t,x))_{1\leq i \leq n}$, the resulting moving-boundary cross-diffusion system reads as
\begin{equation}
    \label{eq:cross-diff-BE}
    \partial_t \bbc(t,x) - \partial_x (\bbA_s(\bbc)\partial_x \bbc)(t,x)  = 0, ~ \mbox{for } t>0 \mbox{ and } x \in (0,X(t)),
\end{equation}
for some cross-diffusion matrix mapping $\bbA_s: \mathbb{R}^n \to \mathbb{R}^{n\times n}$ describing the diffusion in the solid phase, together with the boundary conditions
\begin{align}
\label{eq:bc-BE-1}
 (\bbA_s(\bbc)\partial_x \bbc)(t,0)   & = 0,  \\\label{eq:bc-BE-2}
 (\bbA_s(\bbc)\partial_x \bbc)(t,X(t)) + X'(t)\bbc(t,X(t))  & = \bbF(t), 
\end{align}
and appropriate initial conditions. In other words, no-flux boundary conditions are assumed on the bottom ($x=0$) part of the thin film layer while \eqref{eq:bc-BE-2} expresses the fact that the flux of the $i^{th}$ species absorbed on the upper part of the layer (corresponding to $x = X(t)$) is given by $F_i(t)$. The evolution of the thickness of the layer is assumed to be driven by the following equation
\begin{equation}
    \label{eq:interface-BE}
      X'(t) =  \sum_{i=1}^n F_i(t).
\end{equation}
In~\cite{bakhta2018,bakhta2021}, the absorbed fluxes $\bbF(t)$ are assumed to be explicitly known, which is not realistic since the values of these fluxes depend on the interaction between the gaseous and the solid phase in the hot chamber. This work is a first attempt to build a more evolved model taking into account the evolution of the gaseous phase and its interaction with the solid phase. For the sake of simplicity, we only consider here an isolated system (no incoming fluxes in the hot chamber) in order to mainly focus on the moving-interface coupling. The present paper is then devoted to some theoretical and numerical analysis of the proposed system.

\medskip

Let us present some related contributions from the literature before highlighting the novelty of the present work. Cross-diffusion systems have gained significant interest from the mathematical community in the last twenty years. Indeed, it has been understood in the seminal works~\cite{chen2006,burger2010,jungel2012,jungel2013,jungel2015} that many of these systems have a variational entropy structure, which enables to obtain appropriate estimates in order to prove the existence of weak solutions and to study convergence to equilibrium. In particular, many contributions study theoretical and numerical aspects of the Stefan-Maxwell system \cite{bothe2011,boudin2012,jungel2013,boudin2015maxwell,herberg2017,cancesFiniteVolumesStefan2024} and of the size-exclusion system \cite{burger2010,jungel2015,cancesConvergentEntropyDiminishing2020, hopf2022} that we both consider in this work as typical applications. The variational entropy structure was extended to reaction-diffusion systems of mass-action type in \cite{mielke2011}, and to bulk-interface systems in \cite{glitzky2013}, see also the related works \cite{mielke2015, maas2020, disser2020,morgan2022}. However, all these contributions are restricted to a fixed domain. \newline 
Similar problems posed in moving-boundary domains were investigated in previous works by the authors: in \cite{bakhta2018}, the existence of global weak solutions to the system \eqref{eq:cross-diff-BE}-\eqref{eq:bc-BE-1}-\eqref{eq:bc-BE-2} was proved and the long-time asymptotics were studied in the case of constant fluxes $F$. The rapid stabilization of the associated linearized system was studied in \cite{cauvin-vila2023}. A similar moving-boundary parabolic system was introduced and studied in \cite{aiki2009,aiki2010,aiki2013} to model concrete carbonation. In \cite{chainais-hillairet2018}, the authors introduced a finite volume scheme approximating the system, using a rescaling to a fixed domain, and proved its convergence towards a continuous weak solution. Additionally, the long-time regime of the approximated moving-interface was studied in \cite{zurek2019a} (see also the thesis \cite{zurek2019}). In \cite{portegies2010a}, the authors studied a scalar parabolic problem in a one-dimensional moving-boundary domain, using Wasserstein gradient flows methods. This approach was adapted to a more complex model in \cite{merlet2022}. The authors of \cite{bothe2017} studied Stefan-Maxwell reaction-diffusion in moving-boundary domains of arbitrary dimension. Their model, though more complex, is very much related to ours, but they do not seem to include energy dissipation through the interface and are more interested in dynamical aspects than numerical considerations. We also refer to the monograph \cite{pruss2016} for mathematical tools related to quasilinear parabolic problems in moving-boundary domains. 

\medskip

Our work makes the following contributions:
\begin{itemize}
    \item In Section~\ref{sec:model}, we introduce a new moving-interface cross-diffusion model and highlight its variational entropy structure. The latter implies the thermodynamic consistency of the model and lays the foundations for a rigorous mathematical analysis. The stationary states are identified (Proposition~\ref{prop:stationary}) and insights are given concerning the long-time behaviour of the model (Proposition~\ref{prop:stab}).
    \item In Section~\ref{sec:fv}, a structure-preserving finite-volume scheme is introduced to discretize the system. In contrast to the scheme designed in \cite{chainais-hillairet2018}, we do not rescale the system to a fixed domain but rather discretize the moving-interface following a cut-cell approach. 
    \item We present some results of numerical analysis of the scheme in Section~\ref{sec:num-analysis}. We prove the existence of at least one discrete solution to the scheme at each time step and that this solution preserves the entropic structure of the continuous system (Theorem~\ref{thm:existence}). In particular, the procedure proposed here to update the interface and the mesh at each time step preserves the decay of the entropy at the discrete level (Lemma~\ref{lem:dissip}). In addition, the numerical scheme enables to preserve on the discrete level some expected fundamental properties of the solutions of the model, namely non-negativity of the solutions and total mass (Lemma~\ref{lem:structure}).
    \item Numerical results are given in Section~\ref{sec:numerics}, illustrating the properties of the model and the nice properties of the scheme. Relying on these numerical observations, we formulate some conjectures concerning the long-time behaviour of the solutions of the model. 
\end{itemize}

\section{Moving-interface coupled model}
\label{sec:model}

This section is devoted to the presentation and analysis of the continuous model we consider in this work. The model is first broadly presented in Section~\ref{sec:pres} while technical assumptions on the cross-diffusion matrices together with relevant examples are given in Section~\ref{sec:ass}. The entropy structure of the system is formally investigated in Section~\ref{sec:entropy}. Section~\ref{sec:stationary} is devoted to the characterization of the stationary states and to a discussion about the dynamics. Finally, Section~\ref{sec:stab} presents a stability analysis of a simplified ODE model.

\subsection{Presentation of the model}
\label{sec:pres}

Let $(0,1)$ be the physical domain containing both the solid and gaseous phases, and let $Q:= \mathbb{R}_+ \times (0,1)$ be the time-space domain of the problem. For all $t\geq 0$, let $X(t)\in [0,1]$ denote the position at time $t$ of the interface between the two phases. More precisely, at time $t$, the solid phase occupies the domain $(0, X(t))$ and the gaseous phase occupies the domain $(X(t),1)$. We adopt the convention that, if $X(t) = 0$ (respectively $X(t) = 1$), then the domain is entirely composed of a gaseous (respectively solid) phase. We consider $n$ different chemical species represented by their densities of molar concentration. More precisely, for all $i\in \{1, \dots, n\}$, $c_i(t,x) \geq 0$ represents the density of molar concentration of species $i$ at time $t\geq 0$ and position $x\in (0,1)$ and we set $\bbc(t,x) := (c_i(t,x))_{i \in \{1,\dots,n\}}$. We expect that so-called \emph{volume-filling} constraints are satisfied, then for almost all $(t,x)\in Q$, the vector $\bbc(t,x)$ is expected to belong to the set
\begin{equation}
\label{eq:constraints}
\cA := \left\{ (c_1,\dots,c_n) \in \bR_+^{n}, ~ \sum_{i=1}^n c_i = 1 \right\}.
\end{equation}
From a modelling perspective, the volume-filling constraints arise from size exclusion effects in the solid phase and from isobaric assumptions in the gas mixture (see Section~\ref{sec:ass} below). We assume that initial conditions for the model are given such that, at time $t=0$, 
\begin{subequations}
    \label{eq:model}
\begin{equation}\label{eq:init}
X(0) = X^0 \mbox{ and }\bbc(0,x) = \bbc^0(x), \text{ for a.a } x \in (0,1),
\end{equation}
for some $X^0\in (0,1)$ and $\bbc^0(x) \in \cA$ for almost all $x \in (0,1)$. Now, for almost all $(t,x)\in Q$ and all $i\in \{1, \ldots,n\}$, we denote by $J_i(t,x)\in \mathbb{R}$ the molar flux of species $i$ at time $t$ and position $x\in (0,1)$, and set $\bbJ(t,x):= (J_1(t,x), \dots, J_n(t,x))^T$. The local conservation of matter inside the solid and gaseous phase respectively reads as
\begin{equation}
\label{eq:conservation}
\partial_t \bbc + \partial_x \bbJ = 0, \text{ a.e. in } Q.
\end{equation}
Let us also denote by
$$
Q_s:= \left\{ (t,x)\in Q, \; x \in (0, X(t))\right\} \quad \mbox{ and } \quad Q_g:= \left\{ (t,x)\in Q, \; x \in (X(t), 1)\right\}, 
$$
the time-space domain associated to the solid and gaseous phases respectively. Cross-diffusion phenomena are modelled by a diffusion matrix-valued application $\bbA_s: \mathbb{R}^n \to \mathbb{R}^{n\times n}$ (resp. $\bbA_g: \mathbb{R}^n \to \mathbb{R}^{n\times n}$) in the solid (resp. gaseous) phase, as 
\begin{equation}
\begin{aligned}
      \bbJ &= - \bbA_s(\bbc) \partial_x \bbc, ~ \text{a.e. in } ~ Q_s, \\
      \bbJ &= - \bbA_g(\bbc) \partial_x \bbc, ~ \text{a.e. in} ~ Q_g.
\end{aligned}
  \label{eq:cross-diff}
\end{equation}
We require that the diffusion matrix applications $\bbA_s$ and $\bbA_g$ satisfy some assumptions which will be made precise below in Section~\ref{sec:ass}. On the boundary of the full domain $(0,1)$, zero-flux boundary conditions are imposed, i.e.
\begin{equation}
\bbJ(t,0) = \bbJ(t,1) = 0, \mbox{ for a.a. } t\geq 0. 
\end{equation}
To complete the definition of the model, it remains to introduce (i) the evolution of the position of the interface $X(t)$ and (ii) the flux transmission conditions across this interface. To this aim, we use a flux vector $\bbF(t) = (F_i(t))_{i\in \{1,\ldots,n\}}$ which accounts for phase transition mechanisms located at the vicinity of the interface between the solid and gaseous phases. We focus in this work on interface fluxes of Butler-Volmer type. More precisely, we introduce some rescaled reference chemical potentials $\bbmu^{*,s}:=(\mu_i^{*,s})_{i\in \{1,\ldots,n\}}, \bbmu^{*,g}:=(\mu_i^{*,g})_{i\in \{1,\ldots,n\}} \in \R^n$ and define the constants 
\begin{equation}
    \label{eq:beta}
    \beta_i^* := \exp \l(\llbracket \mu_i^* \rrbracket \r), \quad \llbracket \mu_i^* \rrbracket:= \mu_i^{*,g} - \mu_i^{*,s}, \quad ~ \forall i \in \{1,\dots,n\}.
\end{equation}
Then, the vector $\bbF(t)$ is defined for all $t\geq 0$, for all $i \in \{1,\dots,n\}$ by
\begin{equation}
  \label{eq:BV}
  F_i(t) = \begin{cases}
         \sqrt{\beta_i^*}c_i^g(t) -  \frac{1}{\sqrt{\beta_i^*}}c_i^s(t), &\text{ if } X(t) \in (0,1), \\
        0, &\text{ otherwise}.
        \end{cases}  
  \end{equation}
Note that, when $X(t) \in (0,1)$, this is a trivial instance of the law of mass action. Then, the evolution of the location of the interface is defined as, for almost all $t\geq 0$, 
\begin{equation}\label{eq:evolinter}
X'(t) = \sum_{i=1}^n F_i(t).
\end{equation}
Notice that, if there exists $t_0\geq 0$ such that $X(t_0) = 0$ (respectively $X(t_0) = 1$), then $X(t) = 0$ (respectively $X(t) = 1$) for all $t\geq t_0$, and the system boils down to a single cross-diffusion system defined in the whole domain $(0,1)$ with no-flux boundary conditions and diffusion matrix given by $\bbA_g$ (respectively $\bbA_s$). We define
$$
T:= \mathop{\inf}\left\{t\in \mathbb{R}_+, \; X(t)= 0 \; \text{or}\; X(t) =1\right\}
$$
so that $X(t)\in (0,1)$ if and only if $t\in[0, T)$. For all $0 \leq t < T$ and all $i\in\{1, \ldots, n\}$, we define the quantities
\begin{align*}
    c_i^s(t) &:= c_i(t, X(t)^-), \; c_i^g(t) := c_i(t, X(t)^+), \\
    J_i^s(t) &:= J_i(t, X(t)^-), \; J_i^g(t) := J_i(t, X(t)^+),
\end{align*}
and set 
\[\begin{aligned}
    \bbc^s(t) &:= (c_i^s(t))_{i\in \{1,\ldots, n\}}, ~ \bbc^g(t) := (c_i^g(t))_{i\in \{1,\ldots, n\}}, \\
    \bbJ^s(t) &:= (J_i^s(t))_{i\in \{1,\ldots, n\}}, ~ \bbJ^g(t) := (J_i^g(t))_{i\in \{1,\ldots, n\}}. 
\end{aligned} \] 
Then we impose the following transmission conditions across the moving interface 
\begin{equation}
  \label{eq:interface-bc}
\begin{aligned}
    \bbJ^s &= -\bbA_s(\bbc^s) \l( \partial_x \bbc \r)^s, \\
    \bbJ^g &= -\bbA_g(\bbc^g) \l( \partial_x \bbc \r)^g, \\
    -\bbJ^s(t)+ X'(t) \bbc^s(t)  &= -\bbJ^g(t) + X'(t) \bbc^g(t) = \bbF(t). 
\end{aligned}
\end{equation}
\end{subequations}
Note that this implies in particular that for almost all $0 \leq t < T$,
\begin{equation}
    \label{eq:interface-conservation}
    \llbracket\bbJ(t) \rrbracket - X'(t) \llbracket\bbc(t) \rrbracket = 0,
\end{equation}
where 
$$
\llbracket\bbJ(t) \rrbracket = \bbJ^g(t) - \bbJ^s(t) \quad  \mbox{ and } \quad \llbracket\bbc(t) \rrbracket = \bbc^g(t) - \bbc^s(t). 
$$
Let us point out that conservation of matter follows from the local conservation equation \eqref{eq:conservation}, the zero-flux conditions on the fixed boundary and the conservative condition \eqref{eq:interface-conservation}. Indeed, taking into account the discontinuity of the fluxes and concentrations at the interface, it holds that for any $i \in \{1,\ldots,n\}$ and almost all $t\geq 0$,
\begin{equation*}
\begin{aligned}
   \frac{d}{dt} \left(\int_0^1 c_i(t) ~ dx \right) &= \frac{d}{dt} \left( \int_0^{X(t)} c_i(t) ~ dx + \int_{X(t)}^1 c_i(t) ~ dx \right)  \\
   &= \int_0^{X(t)} \partial_t c_i(t) ~ dx + \int_{X(t)}^1 \partial_t c_i(t) ~ dx - X'(t) \llbracket c_i(t)\rrbracket \\
   &= -\int_0^{X(t)} \partial_x J_i(t) ~ dx - \int_{X(t)}^1 \partial_x J_i(t) ~ dx - X'(t) \llbracket c_i(t)\rrbracket \\
   &= \llbracket J_i(t) \rrbracket -X'(t) \llbracket c_i(t)\rrbracket \\
   &= 0,
\end{aligned}
\end{equation*}
where $ \llbracket c_i(t)\rrbracket  = c_i^g(t) - c_i^s(t)$ and $\llbracket J_i(t)\rrbracket = J_i^g(t) - J_i^s(t)$. 

\subsection{Assumptions on cross-diffusion matrices}
\label{sec:ass}

The aim of this section is to summarize the assumptions that $\bbA_s$ and $\bbA_g$ must satisfy for the coupled model presented in the previous section to enjoy the entropy structure that will be highlighted in the next section. 
\medskip
Let us define
$$
\mathcal V_0:= \left\{ \bbz = (z_1,\ldots,z_n)\in \mathbb{R}^n, \; \sum_{i=1}^n z_i = 0 \right\}. 
$$
We make the following two assumptions: for $\alpha \in \{s,g\}$,
\begin{itemize}
    \item[(A1)] for all $\bbc \in \mathcal A \cap (\mathbb{R}_+^*)^n$, $\bbA_\alpha(\bbc)\left( \mathcal V_0\right) \subset \mathcal V_0$. 
    \item[(A2)] There exists $C_\alpha>0$ and $m_1,\dots,m_n \leq 2$ such that for all $\bbc = (c_i)_{i\in \{1, \ldots, n\}}\in \mathcal A \cap (\mathbb{R}_+^*)^n$ and all $\bbz = (z_i)_{i\in \{1,\ldots,n\}} \in \mathcal V_0$, 
    $$
    \bbz^T \bbA_\alpha(\bbc) \bbH(\bbc)^{-1} \bbz \geq C_\alpha \sum_{i=1}^n c_i^{m_i}|z_i|^2,
    $$
    where $|\cdot|$ denotes the Euclidean norm of $\mathbb{R}^n$ and where 
    \[\bbH(\bbc) = \mathrm{diag}\left(\frac{1}{c_1},\dots,\frac{1}{c_n}\right).\]
 \end{itemize}
Let us make a few remarks before giving explicit examples of diffusion matrices satisfying conditions (A1)-(A2). First, let us point out here that, if $\bbA_s$ and $\bbA_g$ are chosen so that (A1) is satisfied, then in the light of \eqref{eq:conservation}-\eqref{eq:cross-diff} this implies, at least on the formal level, that $\displaystyle \sum_{i=1}^n J_i(t,x) = 0$ from which it follows that $\displaystyle \sum_{i=1}^n c_i(t,x) = 1$, for almost all $(t,x)\in Q$. Second, let us mention that condition (A2) implies that the cross-diffusion system associated to a pure (gaseous or solid) phase enjoys an entropy structure in the sense of~\cite{jungel2015}, associated to the logarithmic free energy functional with free energy density defined by
$$
\forall \bbc:=(c_i)_{i\in \{1,\ldots,n\}}\in \mathcal A, \quad h(\bbc) :=  \sum_{i=1}^n c_i \log c_i.$$
Let us point out in particular that $\bbH(\bbc)$ is the Hessian of $h$ at vector $\bbc\in \mathcal A \cap (\mathbb{R}_+^*)^n$. In the following for all $\bbc \in \mathcal A \cap (\mathbb{R}_+^*)^n$, we denote by 
\begin{equation}
\label{eq:mobility}
    \begin{aligned}
        \bbM_\alpha(\bbc) &= \bbA_\alpha(\bbc) \bbH^{-1}(\bbc).
    \end{aligned}
\end{equation}
the so-called \itshape mobility matrix \normalfont of the phase $\alpha$. 

\medskip

We give in the following two typical examples of diffusion matrix applications which satisfy conditions (A1) and (A2). We will use them throughout the rest of the article. 

\medskip

\bfseries Example 1 (solid phase): \normalfont We consider here the diffusion matrix application introduced in~\cite{bakhta2018}. More precisely, for all $i,j\in \{1,\ldots,n\}$, we introduce some cross-diffusion coefficients $\kappa_{ij}^s=\kappa_{ji}^s > 0$ (with $\kappa_{ii}^s=0$ by convention). For all $\bbc = (c_i)_{i\in \{1, \ldots, n\}} \in \mathbb{R}^n$, the diffusion matrix $\bbA_s(\bbc) \in  \mathbb{R}^{n\times n}$ is defined by
\begin{equation}
    \label{eq:size-exclusion-matrix}
    \begin{aligned}
        (\bbA_s)_{ii}(\bbc) &= \sum_{j \neq i} \kappa_{ij}^s c_j, ~ i \in \{1,\dots,n\}, \\
        (\bbA_s)_{ij}(\bbc) &= - \kappa^s_{ij} c_i, ~ i \neq j \in \{1,\dots,n\}.
    \end{aligned}
\end{equation}
First, it can be easily checked that for all $\bbc\in \mathcal A$, $\bbA_s(\bbc)$ satisfies condition (A1). 
Moreover, it can be checked that $\bbA_s$ satisfies condition (A2) with $m_i = 1$ for any $i \in \{1,\dots,n\}$ and we refer the reader to~\cite[Lemma 1]{bakhta2018} for a proof. 

\medskip

\bfseries Example~2 (gaseous phase): \normalfont For all $i,j\in \{1,\ldots,n\}$, we introduce some (inverse) cross-diffusion coefficients $\kappa_{ij}^g=\kappa_{ji}^g > 0$ (with $\kappa_{ii}^g=0$ by convention). In the gaseous phase, the fluxes are implicitly defined via the Stefan-Maxwell linear system \cite{bothe2011}
\begin{equation}
    \widetilde{\bbA}_g(\bbc) \bbJ = - \partial_x \bbc ~ \text{and} ~ \bbJ  \in \mathcal V_0, ~ \text{a.e. in }Q_g,
    \label{eq:Maxwell-Stefan-flux}
\end{equation}
where for all $\bbc = (c_i)_{i\in \{1, \ldots, n\}} \in \R^n$, the diffusion matrix $\check{\widetilde{\bbA}}_g(\bbc) \in  \mathbb{R}^{n\times n}$ is defined by
\begin{equation}
    \label{eq:size-exclusion-matrix-gas}
    \begin{aligned}
        \left(\widetilde{\bbA}_g\right)_{ii}(\bbc) &= \sum_{j \neq i} \kappa_{ij}^g c_j, ~ i \in \{1,\dots,n\}, \\
        \left(\widetilde{\bbA}_g\right)_{ij}(\bbc) &= - \kappa^g_{ij} c_i, ~ i \neq j \in \{1,\dots,n\}.
    \end{aligned}
\end{equation}
Notice that the expression of $\widetilde{\bbA}_g(\bbc)$ is similar to the one of $\bbA_s(\bbc)$.  The matrix $\widetilde{\bbA}_g(\bbc)$ is not invertible in general, but it holds that for all $\bbc \in \mathcal A \cap \left( \mathbb{R}_+^*\right)^n$, $\widetilde{\bbA}_g(\bbc)\left( \mathcal V_0\right) \subset \mathcal V_0$ and the restriction $\widetilde{\bbA}_g(\bbc)|_{\mathcal V_0}$ defines an invertible linear mapping from $\mathcal V_0$ onto $\mathcal V_0$ (see \cite[Section 5]{bothe2011}). As a consequence, for all $\bbc\in \mathcal A \cap (\mathbb{R}_+^*)^n$, there exists a unique matrix $\bbA_g(\bbc)\in \mathbb{R}^{n\times n}$ so that
$$
\bbA_g(\bbc) z = \left\{
\begin{array}{ll}
\widetilde{\bbA}_g(\bbc)|_{\mathcal V_0}^{-1} \bbz & \mbox{ if }\bbz\in \mathcal V_0,\\
0 & \mbox{ if } \bbz \in \left(\mathcal V_0\right)^{\perp}.\\
\end{array}
\right.
$$
The relationship \eqref{eq:Maxwell-Stefan-flux} can then be rewritten as
\begin{equation}
    \bbJ = - \bbA_g(\bbc)\partial_x \bbc, ~ \text{a.e. in } ~ Q_g,
    \label{eq:Maxwell-Stefan-flux-inverted}
\end{equation}
It can then easily be checked that $\bbA_g$ satisfies condition (A1). The proof that it satisfies condition (A2) with $m_i=1$ for any $i \in \{1,\dots,n\}$ can be found in \cite[Lemma 2.4]{jungel2013}.

\begin{remark}[Physical variables]
\label{rk:variables}
In \cite{bakhta2018}, the system in the solid phase is written in terms of volume fraction variables, and the volume-filling constraint originates from size exclusion effects. Since we work here with molar concentrations, we should rather write $\sum_{j=1}^n v_j c_j = 1$ in $(0,X)$, where the $v_j$ are constant molar volumes, but we normalize these constants to one to simplify. In \cite{bothe2011}, the volume-filling constraint in the Stefan-Maxwell model follows from isobaric conditions in the mixture. Let us point out that, although void can be modelled in the solid layer as one particular species accounting for vacancies at the microscopic level and represented by its volume fraction in the continuous limit, the Stefan-Maxwell model, because it is written in terms of molar concentrations of an incompressible mixture, does not address void (or free volume), see for example \cite{otto1997}. 
\end{remark}

\subsection{Entropy structure}
\label{sec:entropy}

The aim of this section is to highlight the entropy structure of the coupled model introduced in Sections~\ref{sec:pres} and \ref{sec:ass}. For any $X \in [0,1]$ and $\bbc : (0,1) \to \R^n$ such that $\bbc(x) \in \mathcal A$ for almost all $x \in (0,1)$, the coupled free energy functional is defined by
\begin{equation}
  \label{eq:free-energy}
  \cH[\bbc,X] = \int_0^X h_s(\bbc) ~ dx + \int_X^1 h_g(\bbc)~ dx,
\end{equation}
with free energy densities given by, for $\alpha \in \{s,g\}$,
\begin{equation}
\label{eq:energy-density}
\forall \bbc = (c_i)_{i\in \{1, \ldots, n\}}\in \mathcal A, \quad   h_{\alpha}(\bbc) :=   \sum_{i=1}^n c_i (\log(c_i)+\mu_i^{*,\alpha}) -c_i +1.
\end{equation}
For all $\bbc := (c_i)_{i\in \{1,\ldots,n\}} \in \mathcal A$, the chemical potentials are defined for $\alpha \in \{s,g\}$ as $\bbmu_\alpha = (\mu_{\alpha,i})_{i \in \{1,\dots,n\}}$ and
\begin{equation}
\label{eq:chemical-pot}
\mu_{\alpha,i}(\bbc) = \partial_{c_i} h_\alpha(\bbc) =  \log(c_i) + \mu_i^{*,\alpha}, \quad \forall i\in \{1, \ldots, n\}. 
\end{equation}
Note that $\partial_x \bbmu_\alpha(\bbc) = \partial_x \log(\bbc)$ does not depend on the phase $\alpha \in \{s,g\}$ and let us define 
$$
\llbracket\bbmu(t)\rrbracket := \left\{
\begin{array}{ll}
\bbmu_g(\bbc^g(t)) - \bbmu_s(\bbc^s(t)), & \mbox{if }t\in [0,T), \\
0, & \mbox{otherwise}.
\end{array}
\right.
$$
Let us now formally time-differentiate the free energy along solutions to \eqref{eq:model}:
\begin{align*}
    \frac{d}{dt} \cH[\bbc,X] &= \int_0^{X} \partial_t \bbc^T \bbmu_s ~ dx + \int_{X}^1 \partial_t \bbc^T \bbmu_g ~ dx - X' \llbracket h_\alpha(\bbc) \rrbracket \\ 
    &= - \int_0^{X} \partial_x \log(\bbc)^T \bbA_s(\bbc) \partial_x \bbc ~ dx - \int_{X}^1 \partial_x \log(\bbc)^T \bbA_g(\bbc) \partial_x \bbc ~ dx \\
    &+ \llbracket \bbJ^T \bbmu \rrbracket - X' \llbracket h_\alpha(\bbc) \rrbracket \\ 
    &= - \int_0^{X} \partial_x \log(\bbc)^T \bbM_s(\bbc) \partial_x \log(\bbc) ~ dx - \int_{X}^1 \partial_x \log(\bbc)^T \bbM_g(\bbc) \partial_x \log(\bbc) ~ dx \\
    &+ \llbracket \l(\bbJ-X' \bbc\r)^T \bbmu \rrbracket - X' \llbracket h_\alpha(\bbc) - \bbc^T \bbmu \rrbracket 
\end{align*}
where we used integration by parts and assumption \eqref{eq:mobility}. Then using the transmission conditions \eqref{eq:interface-bc} and the fact that  $\llbracket h_\alpha(\bbc) - \bbc^T \bbmu \rrbracket=0$ for $\bbc \in \cA$, we obtain the free energy dissipation equality: for almost all $t\geq 0$,
\begin{equation}
  \label{eq:dissip}
  \begin{aligned}
\frac{d}{dt} \cH[\bbc(t),X(t)] &+ \int_0^{X(t)} \partial_x \log(\bbc(t))^T \bbM_s(\bbc(t))\partial_x \log(\bbc(t)) ~ dx \\
&+ \int_{X(t)}^1 \partial_x \log(\bbc(t))^T \bbM_g(\bbc(t))  \partial_x \log(\bbc(t)) ~ dx  + \bbF(t)^T \llbracket\bbmu(t)\rrbracket = 0,
\end{aligned}
\end{equation}
First, since, for $\alpha \in \{s,g\}$, $\bbA_\alpha$ satisfies condition (A2), it holds that, almost everywhere in $Q_\alpha$,
\begin{equation}
    \label{eq:estimation-diffusion}
    \begin{aligned}
    (\partial_x \log(\bbc)^T \bbM_\alpha(\bbc) \partial_x \log(\bbc) \geq C_\alpha \sum_{i=1}^n c_i^{m_i} | \partial_x \log(c_i)|^2 \geq C_\alpha \sum_{i=1}^n \frac{1}{c_i^{2-m_i}} | \partial_x c_i|^2 \geq  C_\alpha |\partial_x \bbc|^2,
    \end{aligned}
\end{equation}
where in the last inequality we used the fact that, for any $i \in \{1,\dots,n\}$, $m_i \leq 2$ and $c_i \leq 1$. Furthermore, for almost all $t\in (0,T)$, the Butler-Volmer fluxes \eqref{eq:BV} can be reinterpreted, using \eqref{eq:chemical-pot}, as, for $i \in \{1,\dots,n\}$, 
\begin{equation}
    \label{eq:BV-var}
    F_i(t) = c_i^g(t) \exp \l(\frac{1}{2}\llbracket \mu_i^*\rrbracket \r) -c_i^s(t) \exp \l( -\frac{1}{2}\llbracket\mu_i^*\rrbracket \r)  = 2 \sqrt{c_i^s(t) c_i^g(t)} \sinh\left(\frac{1}{2} \llbracket\mu_i(t)\rrbracket\right),
\end{equation} 
which guarantees that, for almost any $t \geq 0$,
\[\bbF(t)^T \llbracket\bbmu(t)\rrbracket \geq 0.\]
As a consequence, the free energy is a \emph{Lyapunov functional} of the coupled system, in the sense that for almost all $t\geq 0$,
\[ \frac{d}{dt} \cH[\bbc(t), X(t)] \leq 0.\]
Note that, given the definition of the free energy density \eqref{eq:energy-density}, this property guarantees the preservation of the nonnegativity of the concentrations along the dynamics.

\medskip

Let us now go a step further in the analysis of the structure of the interface fluxes \eqref{eq:BV-var} with respect to the free energy \eqref{eq:free-energy}. In the series of works \cite{mielke2011, glitzky2013, maas2020}, the mass action law was associated to a quadratic gradient structure with respect to $\cH$.  Later, the authors of \cite{adams2011,adams2013a,mielke2014} tried to derive this structure from microscopic systems using large deviations theory. Interestingly, they did not recover the previously known quadratic structure, but discovered a new generalized (non-quadratic) gradient structure. We use this structure in this work. More precisely, let us introduce an auxiliary function, defined on the real line as
\[ \phi(x) = 4\left(\cosh\left(\frac{x}{2}\right)-1\right),  ~ \forall x \in \R. \]
This function is a \emph{dissipation potential}: it is smooth, strictly convex, nonnegative and such that $\phi(0)=0$. Its derivative is given by
\[ \phi'(x) = 2 \sinh\left(\frac{x}{2}\right), ~ \forall x \in \R,  \]
and $\phi'$ is bijective from $\R$ to $\R$. The convex conjugate of $\phi$ (\emph{dual dissipation potential}) is given by (see \cite[Section 5a]{adams2013a})
\[\phi^*(z) = \sup_{x \in \R} \{ xz - \phi(x) \} = 2z \log\left(\frac{z + \sqrt{z^2+4}}{2}\right) - 2\sqrt{z^2+4} + 4, ~ \forall z \in \R.\]
The Fenchel-Young duality states that $z = \phi'(x)$ if and only if
\begin{equation}
\label{eq:FenchelYoung}
\phi(x) + \phi^*(z) = z x,
\end{equation}
so that
\[ \phi^*(\phi'(x)) = x \phi'(x) - \phi(x), ~ x \in \R,\]
which implies, by strict convexity of $\phi$, the fact that $\phi(0)=0$ and that $\phi' : \R \to \R$ is bijective, that $\phi^* \geq 0$ in $\R$ and $\phi^*(z) = 0 $ if and only if $z=0$. 

Let us now remark that the Butler-Volmer fluxes \eqref{eq:BV-var} are related to $\phi$ via the following relationship   
\[ 
\phi'\left( \llbracket \mu_i\rrbracket \right) = \frac{F_i}{\sqrt{c_i^g c_i^s}}, ~ i \in \{1,\dots,n\}.
\]
As a consequence, applying \eqref{eq:FenchelYoung} to $x:= \llbracket \mu_i\rrbracket$, it holds that for any $i \in \{1,\dots,n\}$, 
\begin{equation*}
\sqrt{c_i^s c_i^g} \left( \phi( \llbracket \mu_i\rrbracket) + \phi^*\left(\frac{F_i}{\sqrt{c_i^s c_i^g}} \right) \right) = F_i  \llbracket \mu_i\rrbracket.  
\end{equation*}
Therefore, one can rewrite \eqref{eq:dissip} as 
\begin{equation}
\label{eq:dissipation-strong}
\begin{aligned}
\frac{d}{dt} \cH[\bbc(t),X(t)] &+ \int_0^{X(t)} (\partial_x \log(\bbc(t)))^T \bbM_s(\bbc(t)) \partial_x \log(\bbc(t))\, dx  \\
&+ \int_{X(t)}^1 (\partial_x \log(\bbc(t)))^T \bbM_g(\bbc(t)) \partial_x \log(\bbc(t))\, dx  \\
&+  \sum_{i=1}^n \sqrt{c_i^s(t) c_i^g(t)} \left( \phi( \llbracket \mu_i(t)\rrbracket) + \phi^*\left(\frac{F_i(t)}{\sqrt{c_i^s(t) c_i^g(t)}}\right) \right)  = 0.
\end{aligned}
\end{equation}

This identity is not yet satisfying, since it may degenerate when $c_i^s, c_i^g=0$. To circumvent this issue, we use on the one hand the estimates \eqref{eq:estimation-diffusion} on the diffusion terms, and on the other hand the fact that it holds
\[ \sqrt{c_i^s c_i^g} \left( \phi( \llbracket \mu_i\rrbracket) + \phi^*\left(\frac{F_i}{\sqrt{c_i^s c_i^g}} \right) \right) \geq \sqrt{c_i^s c_i^g} \phi^*\left(\frac{F_i}{\sqrt{c_i^s c_i^g}}\right) \geq \phi^*(F_i) \geq 0. \]
The first inequality follows from the nonnegativity of $\phi$ and the second from the convexity of $\phi^*$ combined with $\phi^*(0)=0$. We have derived the \emph{weak dissipation inequality}: for some $C_s, C_g> 0$,
\begin{equation}
\label{eq:dissipation-weak}
\frac{d}{dt} \cH[\bbc(t),X(t)] + C_s \int_0^{X(t)} |\partial_x \bbc(t)|^2\, dx  + C_g \int_{X(t)}^1 |\partial_x \bbc(t)|^2\, dx  + \sum_{i=1}^n \phi^*(F_i(t)) \leq 0.
\end{equation}

\begin{remark}[Extension of the model]
\label{rk:limits}
We have seen in the derivation of the dissipation equality \eqref{eq:dissip} that in fact a more general equality holds:
\begin{equation*}
  \label{eq:dissip-bis}
  \begin{aligned}
\frac{d}{dt} \cH[\bbc(t),X(t)] &+ \int_0^{X(t)} \partial_x \log(\bbc(t))^T \bbM_s(\bbc(t))\partial_x \log(\bbc(t)) ~ dx \\
&+ \int_{X(t)}^1 \partial_x \log(\bbc(t))^T \bbM_g(\bbc(t))  \partial_x \log(\bbc(t)) ~ dx  \\
&+ \bbF(t)^T \llbracket\bbmu(t)\rrbracket - X'(t) \left\llbracket \pi(\bbc(t))\right\rrbracket=0,
\end{aligned}
\end{equation*}
where we have introduced the thermodynamic pressure, for $\alpha \in \{s,g\}$, 
\begin{equation}
    \label{eq:pressure}
    \pi_\alpha(\bbc) = \bbc^T \bbmu_\alpha(\bbc) - h_\alpha(\bbc),
\end{equation}
and $ \llbracket \pi(\bbc) \rrbracket := \pi_g(\bbc^g) - \pi_s(\bbc^s)$. The term $X'(t) \left\llbracket \pi(\bbc) \right\rrbracket$ happens to be null in our case, but we have identified three different contributions to free energy dissipation: the two first terms account for bulk diffusion; the term $\bbF(t)^T \llbracket\bbmu\rrbracket$ accounts for "reactions" at the interface, driven by a jump of chemical potentials; the last term $X'(t) \llbracket\pi(\bbc)\rrbracket$ accounts for a displacement of the interface driven by a jump of pressure. It is worth noticing that, if the volume-filling constraints were not normalized to the same constant in \eqref{eq:constraints} (which would be physically relevant, since the molar volumes are not expected to be equal in the two phases), then there would be a (constant) nonzero contribution of $\llbracket\pi(\bbc)\rrbracket$ to the dissipation of the free energy. To go further in the modelling, one may question the relevance of the isobaric assumption (or incompressibility) in the context of vapor deposition. This assumption led to the saturation constraint in the gaseous phase, and is fundamental for the Stefan-Maxwell model. Going beyond it would lead us to implement a different model to describe a \emph{compressible} fluid mixture. In this model, the pressure $\pi$ may become a proper time-dependent variable, and equation \eqref{eq:dissip} would suggest defining an evolution law for the location of the interface of the form 
\begin{equation*}
X' ~\in ~ \partial_{\pi} \psi(\bbc^s, \bbc^g, \blue{-} \llbracket\pi\rrbracket),
\end{equation*}
for some function $\psi: \left\{ \begin{array}{ccc}
\mathcal A \times \mathcal A \times \mathbb{R} & \to & \mathbb{R}_+\\
(\bbc_1, \bbc_2, \pi) & \mapsto &\psi(\bbc_1, \bbc_2, \pi)\\
\end{array}
\right.$ that is convex with respect to its last variable to ensure dissipation. This goes beyond the scope of this work.

\end{remark}

\subsection{Stationary states} 
\label{sec:stationary}
One deduces from the weak dissipation inequality \eqref{eq:dissipation-weak} that stationary solutions $(\boldsymbol{\ovc},\ovX)$ must be such that $\boldsymbol{\ovc}$ is equal to a constant vector $\boldsymbol{\ovc}^s:=(\ovc_i^s)_{i\in \{1, \ldots, n\}}\in \mathcal A$ in $(0,\ovX)$ and another constant vector $\boldsymbol{\ovc}^g:=(\ovc_i^g)_{i\in \{1, \ldots, n\}}\in \mathcal A$ in $(\ovX, 1)$. Moreover, if $\ovX \in (0,1)$, $\overline{F}_i:= \sqrt{\beta_i^*} \ovc_i^g - \frac{1}{\sqrt{\beta_i^*}}\ovc_i^s$ should be equal to $0$ for any $i \in \{1,\dots,n\}$. We set $\bbm^0:= \displaystyle \int_0^1 \bbc^0$ (remember that $\bbc^0$ is the initial condition of $\bbc$ given by \eqref{eq:init}), and denote by $m_i^0$ the $i^{th}$ component of $\bbm^0$ for all $i\in\{1, \ldots, n\}$. 
\begin{definition}\label{def:stationary}
A state $(\boldsymbol{\ovc}^s, \boldsymbol{\ovc}^g, \ovX)\in \mathcal A \times \mathcal A \times [0,1]$ is said to be a stationary state of model~\eqref{eq:model} if and only if
\begin{itemize}
\item[(i)] for all $i\in \{1, \ldots,n\}$, $ \ovX \ovc_i^s + (1-\ovX)\ovc_i^g = m_i^0$ (mass conservation);
\item[(ii)] if $\ovX = 0$ (respectively $\ovX = 1$), then $\boldsymbol{\ovc}^s =0$ (respectively $\boldsymbol{\ovc}^g =0$) (convention in the case of a pure phase);
\item[(iii)] if $\ovX\in (0,1)$, then for all $i\in \{1, \ldots,n\}$, $\overline{F}_i:= \sqrt{\beta_i^*} \ovc_i^g -  \frac{1}{\sqrt{\beta_i^*}}\ovc_i^s = 0$ (zero interface flux in the case of two phases). 
\end{itemize}
\end{definition}
We characterize the set of stationary states of \eqref{eq:model} in the sense of Definition~\ref{def:stationary}, as stated in Proposition~\ref{prop:stationary}. 
\begin{proposition}[Stationary states]
  \label{prop:stationary}
Let us assume that $m_i^0 >0$ for all $i\in \{1, \ldots,n\}$. In addition to the trivial pure phase stationary states $(\bbm^0,0,1)$ and $(0,\bbm^0,0)$, we can characterize the set of stationary states of model~\eqref{eq:model} (in the sense of Definition~\ref{def:stationary}) as follows: \newline

\noindent{\bfseries Case~1:} (Indistinguishable phases) If $\beta_i^* = 1$ for all $i\in \{1, \ldots,n\}$, then the set of non-trivial stationary states is equal to the set of vectors of the form $(\bbm^0,\bbm^0,\ovX)$ with $\ovX\in (0,1)$.  \newline

\noindent{\bfseries Case~2:} (Distinguishable phases) If there exists $i_0\in \{1, \ldots, n\}$ such that $\beta_{i_0}^* \neq 1$, then there exists a non-trivial stationary solution (i.e. such that $\ovX \in (0,1)$) if and only if  
  \begin{equation}
      \label{eq:two-phase-condition}
      \min\l(\sum_{i=1}^n m_i^0 \beta_i^*,\sum_{i=1}^n m_i^0 \frac{1}{\beta_i^*}\r) > 1.
  \end{equation}
  In addition, if \eqref{eq:two-phase-condition} is satisfied, uniqueness holds. 
  \end{proposition}

\begin{proof}
Since $\bbc^0(x) \in \cA$ for almost any $x \in (0,1)$, it holds that $\displaystyle \sum_{i=1}^n m_i^0 = 1$. Moreover, it can be easily proved that for all $x\in \mathbb{R}_+^*$, $x + \frac{1}{x}\geq 2$ with equality if and only if $x=1$. As a consequence, it holds that
\[ \sum_{i=1}^n m_i^0 (\beta_i^* + \frac{1}{\beta_i^*}) \geq 2, \]
with equality if and only if $\beta_i^* = 1$ for all $i\in \{1, \ldots, n\}$. We can thus distinguish two cases:

\medskip  

\noindent {\bfseries Case~1:} For all $i\in \{1, \ldots, n\}$, $\beta_i^* = 1$. Then, it holds that 
$$
\sum_{i=1}^n m_i^0 \beta_i^* = \sum_{i=1}^n m_i^0 \frac{1}{\beta_i^*} = 1.
$$
{\bfseries Case~2:} There exists $i_0\in \{1, \ldots, n\}$ such that $\beta_{i_0}^* \neq 1$. Then, it holds that 
\[ \sum_{i=1}^n m_i^0 (\beta_i^* + \frac{1}{\beta_i^*}) > 2. \]

\noindent In this proof, we consider each case separately. 

\medskip

{\bfseries Case~1:} Let $(\boldsymbol{\ovc}^s, \boldsymbol{\ovc}^g, \ovX) \in \mathcal A \times \mathcal A \times (0,1)$ be a non-trivial stationary state of model~\eqref{eq:model} in the sense of Definition~\ref{def:stationary}. Then, from (iii), it holds that for all $i\in \{1, \ldots,n\}$, $\overline{F}_i = \ovc_i^g - \ovc_i^s = 0$, which yields $\boldsymbol{\ovc}^s = \boldsymbol{\ovc}^g$. Now, the mass conservation property (i) implies necessarily that $\boldsymbol{\ovc}^s = \boldsymbol{\ovc}^g = \bbm^0$. Conversely, for any $\ovX \in (0,1)$, $(\bbm_0, \bbm_0, \ovX)$ can be easily checked to be a stationary state of model~\eqref{eq:model}. 

\medskip

{\bfseries Case~2:} Let $(\boldsymbol{\ovc}^s, \boldsymbol{\ovc}^g, \ovX) \in \mathcal A \times \mathcal A \times (0,1)$ be a non-trivial stationary state of model~\eqref{eq:model} in the sense of Definition~\ref{def:stationary}. Let us first prove that for all $i\in \{1, \ldots, n\}$, $\ovc_i^g>0$ and $\ovc_i^s >0$, reasoning by contradiction. Indeed, if for instance there exists $i_0\in \{1, \ldots, n\}$ such that $\ovc_{i_0}^g = 0$, the fact that $\overline{F}_{i_0} = 0$ yields that $\ovc_{i_0}^s =0$ as well. This yields a contradiction with the fact that 
$$
\ovX \ovc_{i_0}^s + (1-\ovX) \ovc_{i_0}^g = m_{i_0}^0 >0. 
$$
Now, we see that (i) and (iii) are satisfied if and only if
\begin{equation}
\label{eq:defcX}
\forall i \in \{1, \ldots, n\}, \quad \ovc_i^g= \beta_i^* \ovc_i^g\; \mbox{ and } \ovc_i^g = \frac{m_i^0}{\beta_i^*\ovX + (1-\ovX)}.
\end{equation}
As a consequence, $(\boldsymbol{\ovc}^s, \boldsymbol{\ovc}^g, \ovX)$ is a stationary state in the sense of Definition~\ref{def:stationary} if and only if $\ovX$ is a solution in $(0,1)$ to 
\begin{equation}\label{eq:solX}
\sum_{i=1}^n \frac{m_i^0}{\beta_i^* \ovX + (1-\ovX)} = 1.
\end{equation}
In addition, for any solution $\ovX \in (0,1)$ to~\eqref{eq:solX}, $\boldsymbol{\ovc}^g$ and $\boldsymbol{\ovc}^s$ are necessarily given by~\eqref{eq:defcX}, which immediately implies that $\displaystyle \sum_{i=1}^n \ovc_i^s = \sum_{i=1}^n \ovc_i^g = 1$ and thus that $\boldsymbol{\ovc}^g$ and $\boldsymbol{\ovc}^s$ belong to $\mathcal A$. It thus remains to characterize the set of solution $\ovX\in (0,1)$ to~\eqref{eq:solX}. Let us introduce
\begin{equation}
\label{eq:varphi}
    \varphi: \left\{\begin{array}{ccc}
    [0,1] & \to & \mathbb{R}\\
    x & \mapsto & \sum_{i=1}^n \frac{m_i^0}{\beta_i^* x + (1-x)} - 1.
\end{array}\right. 
\end{equation}
Then, the function $\varphi$ is $\mathcal C^\infty$ on $[0,1]$ and its first and second-order derivatives are respectively given by
$$
\forall x\in [0,1], \quad \varphi'(x) = - \sum_{i=1}^n \frac{(\beta_i^*-1)m_i^0}{(\beta_i^* x + (1-x))^2} \; \mbox{ and } \varphi''(x) = 2 \sum_{i=1}^n \frac{(\beta_i^*-1)^2m_i^0}{(\beta_i^* x + (1-x))^3}.
$$
Then, the function $\varphi$ enjoys the following properties. First, it can be easily seen that $\varphi''(x)>0$ for all $x\in[0,1]$, the strict positivity stemming from the fact that there exists at least one index $i_0\in \{1, \ldots, n\}$ such that $\beta_{i_0}^* \neq 1$. Hence $\varphi$ is strictly convex. Second, it holds that $\varphi(0) = \sum_{i=1}^n m_i^0 - 1 = 0$. Thus, there exists at least one solution $\ovX \in (0,1)$ to the equation $\varphi(\ovX) = 0$ if and only if 
\begin{equation}\label{eq:condphi}
\varphi(1)>0 \; \mbox{ and }\; \varphi'(0) < 0, 
\end{equation}
and the solution is then unique. The desired result is then obtained by remarking that \eqref{eq:condphi} is equivalent to \eqref{eq:two-phase-condition}. 
\end{proof}

Let us now make some remarks about the dynamics. On the one hand, if \eqref{eq:two-phase-condition} is violated, we expect solutions to converge in finite time to one of the one-phase solutions, depending on which quantity violates the condition. On the other hand, under condition \eqref{eq:two-phase-condition}, we expect the two-phases stationary state to be the only stable one, and the solution to converge exponentially to this state. Note, however, that in our model, this convergence can by no means hold for any initial condition, but at best for \emph{close enough} initial conditions. Indeed, the interface dynamics only depends on the \emph{local} concentrations around the interface $\bbc^s, \bbc^g$. Thus, the interface is not necessarily monotone over time (think of very slow diffusion) and, since the value of $X(t)$ might reach $0$ or $1$ in finite time, the dynamics may get "trapped" in a one-phase solution, even when \eqref{eq:two-phase-condition} holds. Therefore, it seems difficult to predict to which state the dynamics converges for any initial condition, since it certainly does not depend \emph{only} on the quantities involved in condition \eqref{eq:two-phase-condition}. We refer to the numerical results in Section~\ref{sec:numerics}.

\subsection{Stability in a simplified setting}
\label{sec:stab}
We address in this section the stability of the two-phase equilibrium in a simplified model. Assuming that diffusion is infinitely fast in comparison to the interface reactions, our system amounts to a system of ordinary differential equations. More precisely, we may first assume that all the components are uniform in space and still denoted by $\bbc^s,\bbc^g$. Then, the masses are simply defined by $\bbm^s = \bbc^s X, ~\bbm^g = \bbm^0 - \bbm^s = (1-X) \bbc^g$ and the constraints $\bbc^s, \bbc^g \in \cA, X \in (0,1)$ are equivalent to $(\bbm^s,X) \in \cM$ where
\begin{equation}
    \label{eq:constr-mass}
    \cM := \l\{ (\bbm,X) \in \R^{n+1} ~|~ 0 < m_i < m_i^0, ~ \forall i \in \{1,\dots,n\} ; ~ 0 < X < 1; ~ \sum_{i=1}^n m_i = X \r\}.
\end{equation}
The dynamics then reduces to the system of ordinary differential equations
\begin{equation}
    \label{eq:simplified-dynamics}
    \begin{aligned}
    \frac{d}{dt} \bbm^s(t) &=  \bbF(t),
    \end{aligned}
\end{equation}
associated to the free energy 
\begin{equation}
  \label{eq:free-energy-simplified}
  \overline{\cH}[\bbm^s,X] = X \, h_s \l(\bbc^s(\bbm^s)\r) + (1-X) \, h_g \l( \bbc^g(\bbm^s)\r).
\end{equation}
$\overline{\cH}$ is smooth in $(0,m_1^0) \times \dots \times (0,m_n^0) \times (0,1)$ and, imitating the computations of Section~\ref{sec:entropy}, we obtain 
\begin{align*}
\partial_{m_i^s} \overline{\cH} &= X \partial_{m_i^s} h_s + (1-X) \partial_{m_i^s} h_g = X \partial_{m_i^s} c_i^s \partial_{c_i^s} h_s + (1-X) \partial_{m_i^s} c_i^g \partial_{c_i^g} h_g = \mu_i^s - \mu_i^g,
\end{align*}
and 
\begin{align*}
    \frac{\partial \overline{\cH}}{\partial X} = h_s - \sum_{i=1}^n c_i^s \mu_i^s - \l( h_g - \sum_{i=1}^n c_i^g \mu_i^g \r) = \sum_{i=1}^n \l(c_i^g - c_i^s \r).
\end{align*}
Therefore, for any $(\bbm^s,X) \in \cM$, it holds 
\begin{equation}
    \label{eq:critical}
        \nabla_{\bbm^s} \overline{\cH} = - \llbracket \bbmu \rrbracket, \;  \frac{\partial \overline{\cH}}{\partial X} = 0.
\end{equation}
The free energy dissipation equality boils down to
\begin{equation}
\label{eq:dissipation-simplified}
\frac{d}{dt} \overline{\cH}(t)  + \sum_{i=1}^n \sqrt{c_i^s(m_i^s(t)) c_i^g(m_i^g(t))} \left( \phi( \llbracket \mu_i(m_i^s(t))\rrbracket) + \phi^*\left(\frac{F_i(t)}{\sqrt{c_i^s(m_i^s(t)) c_i^g(m_i^g(t))}}\right) \right)  = 0.
\end{equation}
Note that the latter identity refers to a generalized gradient flow formulation of the system \cite[Section 3]{mielkeIntroductionAnalysisGradients2023}. 
In addition, we have the following result:
\begin{proposition}
\label{prop:stab}
Assume that phases are distinguishable (Case 2 of Proposition~\ref{prop:stationary}), that \eqref{eq:two-phase-condition} is satisfied and let $\l(\overline{\bbm}^s,\ovX\r)$ be the unique non-trivial stationary state obtained in Proposition~\ref{prop:stationary}. Then this state is stable for the dynamics \eqref{eq:simplified-dynamics}.
\end{proposition}
\begin{proof}
First, it follows from \eqref{eq:critical}, the point iii) of Definition~\ref{def:stationary} and the definition of the interface fluxes \eqref{eq:BV-var} that $\l(\overline{\bbm}^s,\ovX\r)$ is a critical point of $\overline{\cH}$. Let us now show that $\overline{\cH}$ is strictly convex at $\l(\overline{\bbm}^s,\ovX\r)$, from which it will follow that the stationary state is in fact a strict local minimizer of $\overline{\cH}$ and is therefore stable, since $\frac{d \overline{\cH}}{dt} \leq 0$. Let us define the functions
\[\begin{aligned}
    \phi_i^\alpha(c) &= c(\log(c) + \mu_i^{*,\alpha}) - c + 1, ~ c>0, ~ \alpha \in \{s,g\}, ~ 1 \leq i \leq n, \\
    \psi_i(m_i^s,X) &= X \phi_i^s(c_i^s(m_i^s)) + (1-X) \phi_i^g(c_i^g(m_i^s)), \qquad 1 \leq i \leq n, 
\end{aligned} 
\]
such that
\begin{equation}
\label{eq:sep}
\overline{\cH}(\bbm^s, X) = \sum_{i=1}^n \psi_i(m_i^s,X).
\end{equation}
Let us fix $i \in \{1,\dots,n\}$ and compute, for any $(m_i^s,X) \in (0,m_i^0) \times (0,1)$,
\begin{align*}
\frac{\partial^2 \psi_i }{\partial X^2}(m_i^s,X) =\; & \frac{m_i^s}{X^2} +  \frac{m_i^g}{(1-X)^2} > 0, 
\\
\frac{\partial^2 \psi_i}{\partial (m_i^s)^2}(m_i^s,X) =\; & \frac{1}{m_i^s}  +  \frac{1}{m_i^g} > 0, 
\\
\frac{\partial^2 \psi_i}{\partial m_i^s \partial X}(m_i^s,X) =\; & - \frac{1}{X} - \frac{1}{1-X}.
\end{align*}
Then  
\[
\mathrm{Tr} (D^2 \psi_i) > 0 \text{ and } \det D^2\psi_i = \frac1{X(1-X)} \frac{\left( c_i^s - c_i^g \right)^2}{c_i^s c_i^g} \geq 0.  
\]
Therefore $\psi_i$ is convex for any $i \in \{1,\dots,n\}$ and so is $\overline{\cH}$, according to \eqref{eq:sep}. Besides, because of \eqref{eq:defcX} and the assumption that $\bbbeta^* \neq (1,\dots,1)^T$, at least one of the previous determinant is strictly positive at $\l(\overline{\bbm}^s,\ovX\r)$, which implies that $\overline{\cH}$ is strictly convex at this state. 
\end{proof}

The study of the stability of stationary states in more complex settings will be the object of future research work.

\section{Finite volume scheme}
\label{sec:fv}
This section is devoted to the finite volume approximation of system \eqref{eq:model}. In Section~\ref{sec:mesh}, we introduce a space-time discretization of the domain and some useful notations. The scheme is presented in two steps: in Section~\ref{sec:conservation}, we discretize the conservation laws, while Section~\ref{sec:pp} is devoted to the mesh displacement. 

\medskip 

In this section, we restrict ourselves to the case where the cross-diffusion matrix application for the solid (respectively gaseous) phase is given by Example~1 (respectively Example~2) of Section~\ref{sec:ass}.

\subsection{Discretization}
\label{sec:mesh}
We consider $N \in \N^*$ reference cells of uniform size $\Delta x = \frac{1}{N}$. The $N+1$ edge vertices are denoted by $0=x_{\frac{1}{2}} \leq x_{\frac{3}{2}}\leq \dots\leq x_{N+\frac{1}{2}}=1$. More precisely, $x_{K+\frac{1}{2}} = K\Delta x$ for all $K\in \{0, \ldots, N\}$. We consider a time horizon $T >0$ and a time discretization with mesh parameter $\Delta t$ defined such that $N_T \Delta t = T$ with $N_T \in \N^*$. The concentrations are discretized as $\bbc^p = (c^p_{i,K})_{i \in \{1,\dots,n\},~ K \in \{1,\dots,N\}}$ for $p \in \{0,\dots,N_T\}$. The interface is time-discretized as $X^p$ for $p \in \{0,\dots,N_T\}$, and we denote by $K^p \in \{0, \ldots, N\}$ the lowest integer such that $|x_{K^p+\frac{1}{2}} - X^p| \leq |x_{K+\frac{1}{2}} - X^p|$ for all $K \in \{0, \ldots, N\}$. For all $p\geq 1$, at time $t^{p-1}=(p-1)\Delta t$, the mesh is locally modified around $X^{p-1}$. More precisely, for all $K\in \{1, \ldots, N\}$, we denote by $C_K^{p-1}$ the $K^{th}$ cell of the mesh defined by
$$
C_K^{p-1}:=\begin{cases}
  (x_{K-\frac{1}{2}}, x_{K+\frac{1}{2}}) & \text{if } K < K^{p-1} \mbox{ or } K> K^{p-1} +1,\\
    (x_{K^{p-1}-\frac{1}{2}}, X^{p-1})  & \text{if } K=K^{p-1}, \\
    (X^{p-1}, x_{K^{p-1}+\frac{3}{2}}) & \text{if } K=K^{p-1}+1.
\end{cases}
$$
We refer to the initial configuration in Figure \ref{fig:displacement}, where the interface cell is assumed to be the $K^{th}$ one (instead of $K^{p-1}$) to alleviate the notation. The size of the cell $C_K^{p-1}$ is then denoted by $\Delta_K^{p-1}$ for all $K\in \{1, \ldots, N\}$: 
\begin{equation}
\label{eq:size-cell}
\Delta_K^{p-1} = 
\begin{cases}
  (X^{p-1} - x_{K^{p-1}-\frac{1}{2}}) & \text{if } K=K^{p-1},\\
  (x_{K^{p-1}+\frac{3}{2}}-X^{p-1}) & \text{if } K=K^{p-1}+1, \\
  \Delta x & \text{otherwise.} 
\end{cases}
\end{equation}
With these notations, an initial condition $\bbc^0$ such that $\bbc^0(x) \in \cA$ for almost all $x \in (0,1)$ is naturally discretized as $ c_{i,K}^0 = \frac{1}{\Delta_K^0} \int_{C_K^0} c_i^0 ~ dx $ for any $i\in \{1,\dots,n\}$, $K\in \{1,\dots,N\}$. 
Starting from the knowledge of $(\bbc^{p-1}, X^{p-1})$, our scheme consists in
\begin{itemize}
\item[i)] solving the conservation laws and updating the interface position, leading to $(\bbc^{p,\star},X^p)$, where $\bbc^{p,\star} = (c_{i,K}^{p,\star})_{i\in \{1, \ldots, n\}, K\in \{1, \ldots, N\}}\in (\mathbb{R}^n)^N$ is a set of intermediate cell values of the concentrations and $X^p \in [0,1]$.
\item[ii)] updating the cells of the mesh $(C_K^p)_{K\in \{1, \ldots, N\}}$ and post-processing the interface concentrations into the final values $\bbc^p$.
\end{itemize}

Section~\ref{sec:conservation} describes the scheme corresponding to step (i), while Section~\ref{sec:pp} describes the scheme corresponding to step (ii).

\subsection{First step: conservation laws}
\label{sec:conservation}
 The conservation laws \eqref{eq:conservation} are discretized implicitly as, for $K \in \{1,\dots,N\}, i \in \{1,\dots,n\}$, 
\begin{subequations}
\label{eq:scheme}
\begin{equation}
\label{eq:conservation-discr}
\frac{1}{\Delta t} (\Delta_{K}^{p,\star} c^{p,\star}_{i,K}- \Delta_{K}^{p-1}c_{i,K}^{p-1}) + J^{p}_{i,K+\frac{1}{2}}\left(\bbc^{p,\star}\right) -J^{p}_{i,K-\frac{1}{2}}\left( \bbc^{p,\star}\right) = 0,
\end{equation}
where we have introduced the numerical fluxes $J^{p}_{i,K+\frac{1}{2}}\left(\bbc^{p,\star}\right) $ and $J^{p}_{i,K-\frac{1}{2}}\left( \bbc^{p,\star}\right)$ which will be defined below and the quantity (see the intermediate mesh in Figure~\ref{fig:displacement} where $K := K^{p-1}$)
\begin{equation}
  \label{eq:size-cell-semi-impl}
  \Delta_K^{p,\star} = 
  \begin{cases}
    (X^p - x_{K^{p-1}-\frac{1}{2}}) & \text{if } K=K^{p-1},\\
    (x_{K^{p-1}+\frac{3}{2}}-X^p) & \text{if } K=K^{p-1}+1, \\
    \Delta x & \text{otherwise.} 
  \end{cases}
\end{equation}
We can impose conditions on the time step $\Delta t$ to guarantee that the new position of the interface $X^p$ remains in the interval $(x_{K^{p-1}-\frac{1}{2}}, x_{K^{p-1}+\frac{3}{2}})$. These conditions are made explicit in the next section and we assume that they hold here. The aim of the term $\frac{1}{\Delta t} (\Delta_{K}^{p,\star} c^{p,\star}_{i,K}- \Delta_{K}^{p-1}c_{i,K}^{p-1})$ for $K= K^{p-1}$ in~\eqref{eq:conservation-discr} is to yield the approximation
\begin{align*} 
\frac{d}{dt} \left(\int_{x_{K^{p-1}-\frac{1}{2}}}^{X(t)} c_i(t) \right)_{|_{t=t^p}} 
& \approx \frac{1}{\Delta t} \left(\int_{x_{K^{p-1}-\frac{1}{2}}}^{X(t^p)} c_i(t^p) - \int_{x_{K^{p-1}-\frac{1}{2}}}^{X(t^{p-1})} c_i(t^{p-1}) \right)\\
& \approx \frac{1}{\Delta t} \left(\int_{x_{K^{p-1}-\frac{1}{2}}}^{X^p} c_{i,K^{p-1}}^{p,\star} - \int_{x_{K^{p-1}-\frac{1}{2}}}^{X^{p-1}} c_{i,K^{p-1}}^{p-1} \right)\\
& = \frac{1}{\Delta t} (\Delta_{K^{p-1}}^{p,\star} c^{p,\star}_{i,K^{p-1}}- \Delta_{K^{p-1}}^{p-1}c_{i,K^{p-1}}^{p-1}).  \\
\end{align*}
Similarly, the aim of the term $\frac{1}{\Delta t} (\Delta_{K}^{p,\star} c^{p,\star}_{i,K}- \Delta_{K}^{p-1}c_{i,K}^{p-1})$ for $K= K^{p-1} +1$ in~\eqref{eq:conservation-discr} is to yield an approximation of $\displaystyle \frac{d}{dt} \left(\int_{X(t)}^{x_{K^{p-1}+\frac{3}{2}}} c_i(t) \right)$. 

\medskip

Let us now turn to the definition of the fluxes. It is sufficient to define for all $p\in \mathbb{N}^*$, $i\in\{1, \ldots, n\}$ and $K\in \{0, \ldots N\}$, the map $J_{i, K+\frac{1}{2}}^p(\bbc)$ for any vector $\bbc = (c_{i,K})_{i\in\{1, \ldots, n\}, K\in \{1, \ldots, N\}} \in (\R^n)^N$. Given such a vector $\bbc$, we will make use of the notation $\bbc_i := (c_{i,K})_{K\in \{1, \ldots, N\}} \in \mathbb{R}_+^N$ for all $i\in \{1, \ldots, n\}$, and $\bbc_K:=(c_{i,K})_{i\in \{1, \ldots, n\}}\in \cA$ for all $K\in \{1, \ldots, N\}$.

\medskip

First, the zero-flux conditions on the boundary of the domain $(0,1)$ are discretized by defining for all $\bbc\in (\R^n)^N$,
$$
\forall i\in \{1, \ldots, n\}, \; J^{p}_{i,\frac{1}{2}}(\bbc) = J^{p}_{i,N+\frac{1}{2}}(\bbc) = 0. 
$$
We are thus left with the definition of the fluxes $\bbJ^{p}_{K+\frac{1}{2}}:=\left(J^{p}_{i,K+\frac{1}{2}}\right)_{i\in\{1, \ldots, n\}}$ for all $K\in \{1,\ldots,N-1\}$. To this aim, we need to introduce, for a given $\bbc = (c_{i,K})_{i\in\{1, \ldots, n\}, K\in \{1, \ldots, N\}} \in (\R^n)^N$, the associated set of edge concentrations $\bbc_{K+\frac{1}{2}} = (c_{i,K+\frac{1}{2}})_{i\in\{1, \ldots, n\}}\in \mathbb{R}_+^n$ for all $K\in \{1, \ldots, N-1\}$, defined through a logarithmic mean as 
\begin{equation}\label{eq:u_isig}
c_{i,K+\frac{1}{2}} := \left\{
\begin{array}{ll}
 0 & \mbox{ if } \min(c_{i,K}, c_{i, K+1})\leq 0, \\
 c_{i,K} & \mbox{ if } 0 < c_{i,K}= c_{i, K+1}, \\
 \frac{c_{i,K} - c_{i,K+1}}{\log(c_{i,K}) - \log(c_{i,K+1})} & \mbox{ otherwise}.\\
\end{array}
\right.
\end{equation}
We also need to introduce the finite difference notation, for all $K \in \{1,\dots,N-1\}$, and any $\bbd = (\bbd_K)_{K\in \{1, \ldots, N\}} \in \left( \mathbb{R}^q \right)^N$ with any $q\in \mathbb{N}^*$,
\[ D_{K+\frac{1}{2}} \bbd := \bbd_{K+1} - \bbd_K.
\]
Then, the choice \eqref{eq:u_isig} yields a discrete chain rule: for any $i\in \{1,\dots,n\}$, $K \in \{1,\dots,N-1\}$, if $c_{i,K}, c_{i,K+1} > 0$, then
\begin{equation}
    \label{eq:chain-rule-discr}
    D_{K+\frac{1}{2}} \bbc_i = c_{i,K+\frac{1}{2}}D_{K+\frac{1}{2}} \log(\bbc_i).
\end{equation} 
We define, for $\alpha \in \{s,g\}$, the coefficients $\kappa^{*,\alpha} = \min_{ij} \kappa_{ij}^\alpha >0$, $\ovkappa_{ij}^\alpha = \kappa_{ij}^\alpha - \kappa^{*,\alpha} \geq 0$ and, for all $\bbu = (u_i)_{i\in \{1, \ldots, n\}} \in \mathbb{R}^n$, we define the matrices $\boldsymbol{\ovA}_\alpha(\bbu)\in \mathbb{R}^{n\times n}$ similarly to \eqref{eq:size-exclusion-matrix} by
\begin{equation*}
    \begin{aligned}
        (\boldsymbol{\ovA}_\alpha)_{ii}(\bbu) &= \sum_{j \neq i} \ovkappa_{ij}^\alpha u_j, ~ i \in \{1,\dots,n\}, \\
        (\boldsymbol{\ovA}_\alpha)_{ij}(\bbu) &= - \ovkappa^\alpha_{ij} u_i, ~ i \neq j \in \{1,\dots,n\}.
    \end{aligned}
\end{equation*}
Then, the bulk solid fluxes are defined as follows (similarly to \cite{cancesConvergentEntropyDiminishing2020}): for all $\bbc = (c_{i,K})_{i\in \{1, \ldots, n\}, K\in \{1, \ldots, N\}}\in (\R^n)^N$, all $i\in \{1, \ldots, n\}$ and any $K\in \{1, \ldots, N-1\}$,
\begin{equation*}
    \Delta x J_{i,K+\frac{1}{2}}^{s}(\bbc) := -\kappa^{*,s} D_{K+\frac{1}{2}} \bbc_i - \sum_{j=1}^n \ovkappa_{ij}^s \left( c_{j,K+\frac{1}{2}} D_{K+\frac{1}{2}} \bbc_i- c_{i,K+\frac{1}{2}} D_{K+\frac{1}{2}} \bbc_j \right),
\end{equation*}
which rewrites in compact form as 
 \begin{equation}
 \label{eq:bulk-flux-discr-s} 
    \Delta x \bbJ_{K+\frac{1}{2}}^{s}(\bbc) = - \widehat{\bbA}_s(\bbc_{K+\frac{1}{2}}) D_{K+\frac{1}{2}} \bbc, ~ K \in \{1,\dots,N-1\},
 \end{equation}
 where 
 \begin{equation}
     \label{eq:As-modif}
      \forall \bbu \in \mathbb{R}^n, \quad \widehat{\bbA}_s(\bbu):= \bar{\bbA_s}(\bbu) + \kappa^{*,s} \bbI,
\end{equation}
with $\bbI \in \R^{n \times n}$ the identity matrix. The bulk gas fluxes \eqref{eq:Maxwell-Stefan-flux} are defined similarly as in the scheme proposed in~\cite{cancesFiniteVolumesStefan2024}, introducing first, 
\begin{equation}
     \label{eq:Ag-modif}
      \forall \bbu \in \mathbb{R}^n, \quad \check{\widetilde{\bbA}}_g(\bbu):= \bar{\bbA_g}(\bbu) + \kappa^{*,g} \bbI,
\end{equation}
so that we can define implicitly, for any $\bbc \in (\R^n)^N$,
\begin{equation}
    \label{eq:bulk-flux-discr-g-implicit}
    \Delta x  \check{\widetilde{\bbA}}_g \left( \bbc_{K+\frac{1}{2}}\right) \bbJ^g_{K+\frac{1}{2}}(\bbc) = - D_{K+\frac{1}{2}} \bbc, ~ K \in \{1,\dots,N-1\}
\end{equation} 
We then define for all $\bbc\in (\R^n)^N$ 
\begin{equation}
\bbJ^p_{K+\frac{1}{2}}(\bbc) = \bbJ^s_{K+\frac{1}{2}}(\bbc), \; \forall 1\leq K < K^{p-1}, \mbox{ and } \bbJ^p_{K+\frac{1}{2}}(\bbc) = \bbJ^g_{K+\frac{1}{2}}(\bbc), \; \forall K^{p-1} < K \leq N-1. \
\end{equation}
In the latter formula, $\bbJ^g_{K+\frac{1}{2}}(\bbc)$ is \emph{any} solution to the system \eqref{eq:bulk-flux-discr-g-implicit}, and at this point we do not even claim existence. The well-posedness of the scheme will follow from the \emph{a priori} estimates proved in Lemma~\ref{lem:structure}, using the following lemma.
\begin{lemma}
\label{lem:inversion}
For any $\bbu \in \cA$, $\check{\widetilde{\bbA}}_g(\bbu)$ has positive eigenvalues, lower bounded by $\kappa^{*,g}>0$. In particular, $\check{\widetilde{\bbA}}_g(\bbu)$ is invertible. 
\end{lemma}
\begin{proof}
    Let $\bbu \in \cA \cap (\R_+^*)^n$. Then it follows from Assumption (A2) on $\bbA_g$ that $\bar{\bbA_g}(\bbu)$ can be rewritten as the product
    \[ \bar{\bbA_g}(\bbu) = \bbH(\bbu) \bar{\bbM_g}(\bbu), \]
    where $\bar{\bbM_g}(\bbu)$ is symmetric positive semi-definite, and $\bbH(\bbu) = \mathrm{diag}(\frac{1}{c_1},\dots,\frac{1}{c_n})$. Hence $\bar{\bbA_g}(\bbu)$ is similar to the symmetric positive semi-definite matrix $\bbH(\bbu)^{\frac{1}{2}} \bar{\bbM_g}(\bbu) \bbH(\bbu)^{\frac{1}{2}}$, and is therefore diagonalizable with nonnegative eigenvalues. In consequence $\check{\widetilde{\bbA}}_g(\bbu)$ is diagonalizable with positive eigenvalues lower bounded by $\kappa^{*,g} > 0$. The result follows for any $\bbu \in \cA$, using this uniform lower bound and the continuity of the spectrum with respect to the matrix coefficients.
\end{proof}
Therefore, provided $\bbc \in \cA^N$, the system \eqref{eq:bulk-flux-discr-g-implicit} can be equivalently written
\begin{equation}
    \label{eq:bulk-flux-discr-g} 
    \Delta x  \bbJ^g_{K+\frac{1}{2}}(\bbc) = - \widehat{\bbA}_g(\bbc_{K+\frac{1}{2}}) D_{K+\frac{1}{2}} \bbc, ~ K \in \{1,\dots,N-1\},
\end{equation} 
where $ \widehat{\bbA}_g(\bbc_{K+\frac{1}{2}}) = \check{\widetilde{\bbA}}_g(\bbc_{K+\frac{1}{2}})^{-1}$. 

We are now left with the definition of the interface flux $\bbJ^p_{K^{p-1} + \frac{1}{2}}$. The interface fluxes \eqref{eq:BV} are discretized by introducing for all $\bbc = (c_{i,K})_{i\in \{1, \ldots, n\}, K\in \{1, \ldots, N\}}\in (\R^n)^N$,  
\begin{equation}
    \label{eq:numerical-BV}
    F_i^p(\bbc) = \sqrt{\beta_i^*}c_{i,(K^{p-1}+1)} -  \frac{1}{\sqrt{\beta_i^*}} c_{i,K^{p-1}}, ~ \forall i \in \{1,\dots,n\},
\end{equation}
and we define 
\begin{equation}\label{eq:interfaceflux}
\bbJ_{K^{p-1}+\frac{1}{2}}^{p}(\bbc) := -\bbF^{p}(\bbc),
\end{equation}
where $\bbF^{p}:= \left(F_i^{p}\right)_{i\in \{1, \ldots, n\}}$. This expression stems from the fact that, on the continuous level, it holds that
\begin{align*}
\frac{d}{dt} \left(\int_{x_{K^{p-1}-\frac{1}{2}}}^{X(t)} \bbc(t) \right) &= X'(t) \bbc^s(t) +  \int_{x_{K^{p-1}-\frac{1}{2}}}^{X(t)} \partial_t \bbc(t),\\
& = X'(t) \bbc^s(t) - \bbJ^s(t) + \bbJ(t,x_{K^{p-1}-\frac{1}{2}}),\\
& = \bbF(t)+ \bbJ(t,x_{K^{p-1}-\frac{1}{2}}).
\end{align*}
Finally, \eqref{eq:evolinter} is discretized as
\begin{equation}
\label{eq:interface-discrete}
X^p = X^{p-1} + \Delta t \sum_{i=1}^n F_i^{p}(\bbc^{p,\star}).
\end{equation}
\end{subequations}
A solution to \eqref{eq:scheme} is denoted by $(\bbc^{p,\star},X^p)$. In the sequel, we also make use of the following notation: for all $K\in \{0, \ldots, N\}$, 
$$
\bbJ_{K+\frac{1}{2}}^{p,\star}:= (\bbJ_{i,K+\frac{1}{2}}^{p,\star})_{i\in \{1, \ldots, n\}} = \bbJ^p_{K+\frac{1}{2}}(\bbc^{p,\star}) \quad \mbox{ and } \quad \bbF^{p,\star}:= (F_i^{p,\star})_{i\in \{1, \ldots, n\}} = \bbF^p(\bbc^{p,\star}).
$$

\subsection{Post-processing}
\label{sec:pp}
Once the new value of the interface location $X^p$ has been determined, the updated value of the integer $K^p$ can be computed as well. If applicable, the mesh has then to be updated, together with the discretized values of the concentrations accordingly. 

\medskip

First note that, if for all $1\leq K \leq N$, $\bbc^{p,*}_{K} := \left(c^{p,*}_{i,K}\right)_{i\in \{1, \ldots, n\}}\in \mathcal A$ (we prove in Lemma~\ref{lem:structure} below that it is indeed the case), this implies the uniform bound on the interface fluxes
\begin{equation}
    \label{eq:interface-potentials-infty}
|F_i^{p,\star}| \leq 2 |\cosh(\frac{1}{2} \llbracket\mu_i^{*} \rrbracket)|, ~  \forall i \in \{1,\dots,n\}. 
\end{equation}
Therefore, we obtain from \eqref{eq:interface-discrete}, defining 
\begin{equation}
\label{eq:constant-CFL}
    C_{\bbmu^*}:=\max_{i\in \{1, \ldots, n\}} 2 |\cosh(\frac{1}{2} \llbracket \mu_i^{*} \rrbracket)| >0,
\end{equation} for all $p\in \mathbb{N}^*$, 
\[ |X^p-X^{p-1}| \leq C_{\bbmu^*} \Delta t. \]
Assuming then that $\Delta t>0$ is chosen in order to ensure the condition
\begin{equation}
\label{eq:CFL}
\Delta t \leq \frac{\Delta x}{2C_{\bbmu^*}},
\end{equation}
we obtain that, necessarily, for all $p\in\mathbb{N}^*$, $|X^p-X^{p-1}| \leq \frac{1}{2} \Delta x$, which in particular ensures that $|K^p-K^{p-1}| \leq 1$ and that $X^p$ belongs to $(x_{K^{p-1}-\frac{1}{2}}, x_{K^{p-1}+\frac{3}{2}})$ (see~Figure~\ref{fig:displacement} with $K:=K^{p-1}$).

\medskip

If $K^p=K^{p-1}$, then we can directly iterate the scheme with $\bbc^p = \bbc^{p,\star}$. Otherwise, let us assume that $K^p = K^{p-1}+1 < N$, the case $K^p=K^{p-1}-1 > 1$ being treated similarly. We perform the following steps (see the final mesh in Figure~\ref{fig:displacement} where the notation $K := K^{p-1}$ is used):
\begin{itemize}
\item[i)] \emph{Projection:} The value $c_{i,K^{p-1}}^{p,\star}$ is assigned to the virtual cell $(x_{K^{p-1}-\frac{1}{2}},X^p)$. We assign this value to both the fixed cell $C_{K^{p-1}}^p =(x_{K^{p-1}-\frac{1}{2}},x_{K^{p-1}+\frac{1}{2}})$ and the new interface cell $C_{K^{p-1}+1}^p= C_{K^p}^p=(x_{K^{p-1}+\frac{1}{2}},X^p)$:
\begin{equation}
  \label{eq:projection}
c_{i,K^{p-1}}^p = c_{i,K^{p-1}+1}^p := c_{i,K^{p-1}}^{p,\star}.
\end{equation}
\item[ii)] \emph{Average:} We define the value in the cell $C^p_{K^{p-1}+2}=(X^p,x_{K^{p-1}+2})$ as the following average:
\begin{equation}
  \label{eq:average}
  c_{i,K^{p-1}+2}^p =  c_{i,K^{p}+1}^p := \frac{1}{\Delta x + \Delta_{K^{p}}^{p,\star}} \left[ \Delta_{K^{p}}^{p,\star} c_{i,K^{p}}^{p,\star} + \Delta x ~ c_{i,K^{p}+1}^{p,\star} \right].
\end{equation}
\item[iii)] For all $1\leq  K \leq N$, such that $K\neq K^{p-1}, K^{p-1}+1, K^{p-1}+2$,  $c_{i,K}^p = c_{i,K}^{p,\star}$. 
\end{itemize} 
In the limit cases where $K^p=N$ (resp. $K^p=1$), we consider in agreement with the continuous model that only a single phase remains in the system, and definitively set $X^p=1$ (resp. $X^p = 0$). 

\medskip

The scheme \eqref{eq:scheme}-\eqref{eq:projection}-\eqref{eq:average} is now complete and referred to as $(S)$. 
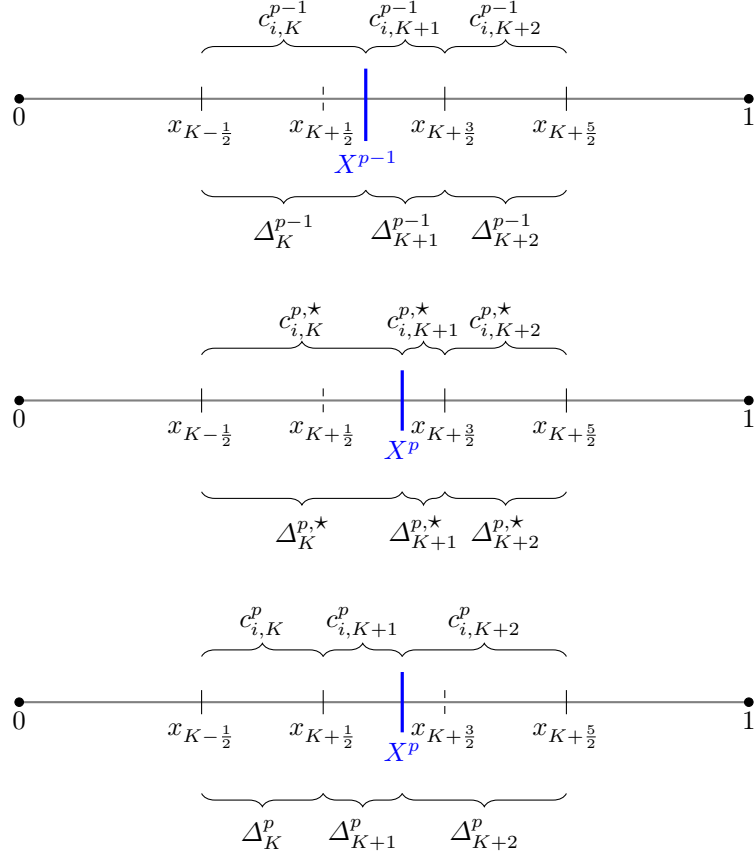
\begin{figure}
  \centering
\begin{tikzpicture}[scale=0.8]
\tikzmath{\xl=0; \xr = 12; \y=-5; \yf=-10;}

\draw[gray, thick] (\xl,0) -- (\xr,0);
\filldraw[black] (\xl,0) circle (2pt) node[below]{0};
\filldraw[black] (\xr,0) circle (2pt) node[below]{1};
\draw[dashed] (5,0.2) -- (5,-0.2) node[below]{$x_{K+\frac{1}{2}}$}; 
\draw (3,+0.2) -- (3,-0.2) node[below]{$x_{K-\frac{1}{2}}$};
\draw (7,0.2) -- (7,-0.2) node[below]{$x_{K+\frac{3}{2}}$};
\draw (9,0.2) -- (9,-0.2) node[below]{$x_{K+\frac{5}{2}}$};
\draw[very thick,blue] (5.7,0.5) -- (5.7,-0.7) node[below]{$X^{p-1}$};
\draw [decorate,decoration={brace,amplitude=5pt,raise=4ex}]
  (3,0) -- (5.7,0) node[midway, yshift=3em]{$c_{i,K}^{p-1}$};
\draw [decorate,decoration={brace,amplitude=5pt,raise=4ex}]
  (5.7,0) -- (7,0) node[midway, yshift=3em]{$c_{i,K+1}^{p-1}$};
\draw [decorate,decoration={brace,amplitude=5pt,raise=4ex}]
  (7,0) -- (9,0) node[midway, yshift=3em]{$c_{i,K+2}^{p-1}$};
\draw [decorate,decoration={brace,amplitude=5pt,raise=8ex}]
  (5.7,0) -- (3,0) node[midway, yshift=-5em]{$\Delta_K^{p-1}$};
\draw [decorate,decoration={brace,amplitude=5pt,raise=8ex}]
(7,0) -- (5.7,0) node[midway, yshift=-5em]{$\Delta_{K+1}^{p-1}$};
\draw [decorate,decoration={brace,amplitude=5pt,raise=8ex}]
  (9,0) -- (7,0) node[midway, yshift=-5em]{$\Delta_{K+2}^{p-1}$};
\draw[gray, thick] (\xl,\y) -- (\xr,\y);
\filldraw[black] (\xl,\y) circle (2pt) node[below]{0};
\filldraw[black] (\xr,\y) circle (2pt) node[below]{1};
\draw[dashed] (5,\y+0.2) -- (5,\y-0.2) node[below]{$x_{K+\frac{1}{2}}$}; 
\draw (3,\y+0.2) -- (3,\y-0.2) node[below]{$x_{K-\frac{1}{2}}$};
\draw (7,\y+0.2) -- (7,\y-0.2) node[below]{$x_{K+\frac{3}{2}}$};
\draw (9,\y+0.2) -- (9,\y-0.2) node[below]{$x_{K+\frac{5}{2}}$};
\draw[very thick,blue] (6.3,\y+0.5) -- (6.3,\y-0.5) node[below]{$X^{p}$};
\draw [decorate,decoration={brace,amplitude=5pt,raise=4ex}]
  (3,\y) -- (6.3,\y) node[midway, yshift=3em]{$c_{i,K}^{p,\star}$};
\draw [decorate,decoration={brace,amplitude=5pt,raise=4ex}]
  (6.3,\y) -- (7,\y) node[midway, yshift=3em]{$c_{i,K+1}^{p,\star}$};
\draw [decorate,decoration={brace,amplitude=5pt,raise=4ex}]
  (7,\y) -- (9,\y) node[midway, yshift=3em]{$c_{i,K+2}^{p,\star}$};
\draw [decorate,decoration={brace,amplitude=5pt,raise=8ex}]
(6.3,\y) -- (3,\y) node[midway, yshift=-5em]{$\Delta_K^{p,\star}$};
\draw [decorate,decoration={brace,amplitude=5pt,raise=8ex}]
(7,\y) -- (6.3,\y) node[midway, yshift=-5em]{$\Delta_{K+1}^{p,\star}$};
\draw [decorate,decoration={brace,amplitude=5pt,raise=8ex}]
  (9,\y) -- (7,\y) node[midway, yshift=-5em]{$\Delta_{K+2}^{p,\star}$};

\draw[gray, thick] (\xl,\yf) -- (\xr,\yf);
\filldraw[black] (\xl,\yf) circle (2pt) node[below]{0};
\filldraw[black] (\xr,\yf) circle (2pt) node[below]{1};

\draw (5,\yf+0.2) -- (5,\yf-0.2) node[below]{$x_{K+\frac{1}{2}}$}; 
\draw (3,\yf+0.2) -- (3,\yf-0.2) node[below]{$x_{K-\frac{1}{2}}$};
\draw[dashed] (7,\yf+0.2) -- (7,\yf-0.2) node[below]{$x_{K+\frac{3}{2}}$};
\draw (9,\yf+0.2) -- (9,\yf-0.2) node[below]{$x_{K+\frac{5}{2}}$};
\draw[very thick,blue] (6.3,\yf+0.5) -- (6.3,\yf-0.5) node[below]{$X^{p}$};

\draw [decorate,decoration={brace,amplitude=5pt,raise=4ex}]
  (3,\yf) -- (5,\yf) node[midway, yshift=3em]{$c_{i,K}^{p}$};
\draw [decorate,decoration={brace,amplitude=5pt,raise=4ex}]
  (5,\yf) -- (6.3,\yf) node[midway, yshift=3em]{$c_{i,K+1}^{p}$};
\draw [decorate,decoration={brace,amplitude=5pt,raise=4ex}]
  (6.3,\yf) -- (9,\yf) node[midway, yshift=3em]{$c_{i,K+2}^{p}$};
  \draw [decorate,decoration={brace,amplitude=5pt,raise=8ex}]
  (5,\yf) -- (3,\yf) node[midway, yshift=-5em]{$\Delta_K^{p}$};
  \draw [decorate,decoration={brace,amplitude=5pt,raise=8ex}]
  (6.3,\yf) -- (5,\yf) node[midway, yshift=-5em]{$\Delta_{K+1}^{p}$};
  \draw [decorate,decoration={brace,amplitude=5pt,raise=8ex}]
  (9,\yf) -- (6.3,\yf) node[midway, yshift=-5em]{$\Delta_{K+2}^{p}$};
\end{tikzpicture}
\caption{A virtual mesh displacement between $t^{p-1} = (p-1) \Delta t$ and $t^p = p \Delta t$, where $K := K^{p-1}$.}
\label{fig:displacement}
\end{figure}

\section{Elements of numerical analysis}
\label{sec:num-analysis}
The aim of this section is to gather some elements of numerical analysis of the finite volume scheme presented in Section~\ref{sec:fv}. We present here some properties of the scheme on a fixed grid, the convergence of the scheme when discretization parameters go to zero being work in progress. From now on and in all the rest of the section, we assume that the time step $\Delta t$ satisfies the following assumption: 
\begin{equation}\label{eq:dtass}
  \Delta t < \frac{\Delta x}{2C_{\bbmu}^*},
\end{equation}
where $C_{\bbmu}^*$ was defined in \eqref{eq:constant-CFL}.

\subsection{``Modified'' scheme (\~S)}
\label{sec:fv-modified}

In this section, we introduce some modifications to the scheme that are helpful to obtain the desired \emph{a priori} estimates. Indeed, summing the conservation laws \eqref{eq:conservation-discr} over the number of species does not allow to easily prove the volume-filling property because, in contrast to the situation on the continuous level, the discrete interface fluxes \emph{do not vanish}. We overcome this difficulty by modifying \eqref{eq:interfaceflux}, defining seemingly non-conservative discrete interface fluxes. Second, proving nonnegativity of the solution is not obvious because of the lack of sign of the interface fluxes \eqref{eq:numerical-BV}, so we introduce some truncations in these quantities. Keeping in mind the need for uniform bounds for \eqref{eq:interface-potentials-infty} to hold, one should also introduce suitable normalizations. We then look for a continuous map that $\mathbb{R}^n \ni x \mapsto x^\diamond$ that enjoys the following properties   
\begin{itemize}
\item[(P1)] for all $x\in \mathbb{R}^n$, $0 \leq x^\diamond_i \leq 1, ~ \forall i \in \{1,\dots,n\}$;
\item[(P2)] for all $x\in \mathcal A$, $x^\diamond = x$; 
\item[(P3)] for all $x := (x_i)_{i\in\{1, \ldots, n\}}\in \mathbb{R}^n$ such that $\sum_{i=1}^n x_i = 1$, denoting by $(x_i^\diamond)_{i\in \{1, \ldots, n\}}$ the coordinates of $x^\diamond$, for any $i\in \{1, \ldots, n\}$, $x_{i}^\diamond = 0$ if and only if $x_{i}\leq 0$.
 \end{itemize}

 An example of such continuous map $\mathbb{R}^n \ni x \mapsto x^\diamond$ is given by 
\begin{equation}
\label{eq:defdiamond}
    \forall i \in \{1, \ldots, n\}, \quad x_i^\diamond = \frac{x_i^+}{\max\l(1, \sum_{j=1}^n x_j^+\r)}.
\end{equation}

We then introduce the following modified scheme. Starting from $(\bbc^{p-1}, X^{p-1}) \in \mathcal A^N \times (0,1)$, and assuming that $1 < K^{p-1} < N$, we first compute $(\bbc^{p,\star}, X^p)\in (\mathbb{R}^n)^N \times [0,1]$ solution to \eqref{eq:scheme} up to the following modifications:
\begin{itemize}
    \item[(i)] The discrete interface fluxes \eqref{eq:numerical-BV} are replaced by  $\widetilde{\bbF}^p(\bbc) = \l( \widetilde{F}^p_i(\bbc) \r)_{i \in \{1,\dots,n\}}$ with, for all $i\in \{1,\ldots,n\}$,  
\begin{equation}
    \label{eq:numerical-BV-trunc}
   \widetilde{F}^p_i(\bbc) = \sqrt{\beta_i^*}c_{i,(K^{p-1}+1)}^\diamond  -  \frac{1}{\sqrt{\beta_i^*}} c_{i,K^{p-1}}^\diamond. 
\end{equation}
We will also use the notation $\widetilde{\bbF}^{p, \star} = (\widetilde{F}_i^{p, \star})_{i\in \{1, \ldots, n\}} = \widetilde{\bbF}^p(\bbc^{p,\star})$.
\item[(ii)] Equation \eqref{eq:interfaceflux} is replaced by two different equations on each respective side of the interface: for all $\bbc = (c_{i,K})_{i\in \{1, \ldots, n\}, K\in \{1, \ldots, N\}} \in (\mathbb{R}^n)^N$, 
\begin{equation}
    \label{eq:interfaceflux-mod}
    \begin{aligned}
        \bbJ_{K^{p-1}+\frac{1}{2}}^{p, s}(\bbc) &= -\l(\sum_{i=1}^n c_{i,K^{p-1}} \r) \widetilde{\bbF}^p(\bbc), \\
        \bbJ_{K^{p-1}+\frac{1}{2}}^{p, g}(\bbc) &= -\l(\sum_{i=1}^n c_{i,K^{p-1}+1} \r) \widetilde{\bbF}^p(\bbc).
    \end{aligned}
\end{equation}
We then define $\bbJ_{K^{p-1}+\frac{1}{2}}^{p,\star,s} =  \bbJ_{K^{p-1}+\frac{1}{2}}^{p, s}(\bbc^{p,\star})$ and $\bbJ_{K^{p-1}+\frac{1}{2}}^{p,\star,g} =  \bbJ_{K^{p-1}+\frac{1}{2}}^{p, g}(\bbc^{p,\star})$. This leads to the (seemingly non-conservative) scheme, for $i \in \{1,\dots,n\},$
\begin{equation}
\begin{aligned}
 \frac{\widetilde{X}^{p}-x_{K^{p-1}-\frac{1}{2}}}{\Delta t}\l(c^{p,\star}_{i,K^{p-1}} -c_{i,K^{p-1}}^{p-1}\r) - \l(\sum_{j=1}^n c_{j,K^{p-1}}^{p,\star} \r) \widetilde{F}_i^{p,\star} - J_{i,K^{p-1}-\frac{1}{2}}^{p,\star} = 0, \\
 \frac{x_{K^{p-1}+\frac{3}{2}}-\widetilde{X}^{p}}{\Delta t}\l(c^{p,\star}_{i,K^{p-1} + 1} -c_{i,K^{p-1} + 1}^{p-1}\r) + J^{p,\star}_{i,K^{p-1}+\frac{3}{2}} + \l(\sum_{j=1}^n c_{j,K^{p-1}+1}^{p,\star} \r)\widetilde{F}_i^{p,\star} = 0, \\
 \frac{\Delta x}{\Delta t}\l(c^{p,\star}_{i,K} -c_{i,K}^{p-1}\r) + J^{p,\star}_{i,K+\frac{1}{2}} - J^{p,\star}_{i,K-\frac{1}{2}} = 0, ~ \forall K \notin \{K^{p-1}, K^{p-1}+1\}.
\end{aligned}
\end{equation}
where have accordingly modified the interface evolution as
\begin{equation}
    \label{eq:interface-normalized}
    \widetilde{X}^{p} := X^{p-1} + \Delta t  \sum_{i=1}^n \widetilde{F}_i^{p,\star}.
\end{equation}
We also introduce the notation
\begin{equation}
\label{eq:size-cell-norm}
\widetilde{\Delta}_K^{p} = 
\begin{cases}
  (\widetilde{X}^{p} - x_{K^{p-1}-\frac{1}{2}}) & \text{if } K=K^{p-1},\\
  (x_{K^{p-1}+\frac{3}{2}}-\widetilde{X}^{p}) & \text{if } K=K^{p-1}+1, \\
  \Delta x & \text{otherwise,} 
\end{cases}
\end{equation}
\end{itemize}

The resulting modified scheme is referred to as $(\widetilde{S})$, while we still denote a possible solution by $\displaystyle (\bbc^{p,\star},X^p)$ (respectively $\displaystyle (\bbc^{p},X^p)$ after the post-processing step) for the sake of simplicity. We will prove existence of a solution to $(\widetilde{S})$ that satisfies the positivity of the concentrations and the volume-filling constraint, and is therefore a solution to the original scheme $(S)$.

\subsection{Non-negativity and volume-filling constraints}
The following lemma provides some \emph{a priori} estimates on any solution to $(\widetilde{S})$.
\begin{lemma}
  \label{lem:structure}
 Let $p\in \mathbb{N}\setminus \{0\}$. Let $(\bbc^{p-1},X^{p-1})$ be such that for all $1\leq K \leq N$, $\bbc_{K}^{p-1}:= (c_{i,K}^{p-1})_{i\in \{1, \ldots, n\}}$ belongs to $\mathcal A$ and $X^{p-1} \in [0,1]$. Let $(\bbc^{p},X^{p})$ be a solution to $(\widetilde{S})$. Then it holds that
\begin{align*}
  c_{i,K}^p & \geq 0, \; \forall i\in \{1, \ldots, n\}, \; \forall K\in \{1, \ldots, N\},\\
   \sum_{i=1}^n c_{i,K}^p &= 1, \; \forall K \in \{1,\dots,N\}, \\
  \sum_{K=1}^N \Delta_K^p c_{i,K}^p &= \sum_{K=1}^N \Delta_K^{p-1} c_{i,K}^{p-1}, \; \forall i \in \{1,\dots,n\}.
\end{align*}
In particular, $\bbJ_{K^{p-1}+\frac{1}{2}}^{p,\star,s}=\bbJ_{K^{p-1}+\frac{1}{2}}^{p,\star,g}$ and for all $i \in \{1,\dots,n\}$ and all $K\in \{1,\dots,N\}$, $\left(c_{i,K}^{p,\star}\right)^\diamond=c_{i,K}^{p,\star}$, from which It follows that $(\bbc^{p},X^p)$ is also a solution to $(S)$. In addition, the Stefan-Maxwell linear system \eqref{eq:bulk-flux-discr-g-implicit} can be inverted. 
\end{lemma}

\begin{proof}
Let $\displaystyle (\bbc^{p,\star},X^p)$ be a solution to ($\widetilde{S}$) before post-processing. Using formulas \eqref{eq:numerical-BV-trunc}-\eqref{eq:interface-normalized}-\eqref{eq:size-cell-norm} together with condition \eqref{eq:dtass}, it holds
\[ 
\begin{aligned}
    \widetilde{\Delta}_{K^{p-1}}^{p} &= \Delta_{K^{p-1}}^{p-1} + \Delta t \sum_{i=1}^n \widetilde{F}_i^{p,\star}>0, \\
    \widetilde{\Delta}_{K^{p-1}+1}^{p} &= \Delta_{K^{p-1}+1}^{p-1} - \Delta t \sum_{i=1}^n \widetilde{F}_i^{p,\star} > 0.    
\end{aligned}
\]
Let us first prove that for all $K\in \{1, \ldots, N\}$, $\sum_{i=1}^n c_{i,K}^{p,\star} = 1$. We sum the conservation laws \eqref{eq:conservation-discr} over $i\in \{1, \ldots, n\}$. On the one hand, when summing the solid bulk fluxes, the cross-diffusion terms disappear and we obtain linear diffusion associated to the parameter $\kappa^{*,s}$. On the other hand, the interface fluxes vanish, thanks to the modification \eqref{eq:interfaceflux-mod}. We obtain in $C_{K^{p-1}}^{p-1}$:
\begin{align*}
    &\frac{1}{\Delta t}  \left( \widetilde{\Delta}_{K^{p-1}}^{p}\sum_{i=1}^n c_{i,K^{p-1}}^{p,\star} - \Delta_{K^{p-1}}^{p-1}\sum_{i=1}^n c_{i,K^{p-1}}^{p-1}\right) - \l( \sum_{j=1}^n c_{j,K^{p-1}}^{p,\star} \r)  \l( \sum_{i=1}^n 
\widetilde{F}_i^{p,\star} \r) - \sum_{i=1}^n J_{i,K^{p-1}-\frac{1}{2}}^{p,\star} \\
    & =  \frac{\Delta_{K^{p-1}}^{p-1}}{\Delta t} \left(\sum_{i=1}^n c_{i,K^{p-1}}^{p,\star} - 1\right) + \kappa^{*,s} \sum_{i=1}^n D_{K^{p-1}-\frac{1}{2}} \bbc_i^{p,\star}.
\end{align*}
Similarly, it holds in $C_{K^{p-1}+1}^{p-1}$ (since for any $\bbc \in \cA$, $\text{Ran}(\boldsymbol{\ovA}_g(\bbc)) \subset \cV_0 $): 
\begin{align*}
    &\frac{1}{\Delta t}  \left( \widetilde{\Delta}_{K^{p-1}+1}^{p}\sum_{i=1}^n c_{i,K^{p-1}+1}^{p,\star} - \Delta_{K^{p-1}+1}^{p-1}\sum_{i=1}^n c_{i,K^{p-1}+1}^{p-1}\right) + \l( \sum_{j=1}^n c_{j,K^{p-1}+1}^{p,\star} \r)  \l( \sum_{i=1}^n 
\widetilde{F}_i^{p,\star} \r) \\ 
 &+ \sum_{i=1}^n J_{i,K^{p-1}+\frac{3}{2}}^{p,\star} 
     =  \frac{\Delta_{K^{p-1}}^{p-1}}{\Delta t} \left(\sum_{i=1}^n c_{i,K^{p-1}+1}^{p,\star} - 1\right) - \l(\kappa^{*,g} \r)^{-1} \sum_{i=1}^n D_{K^{p-1}+\frac{3}{2}} \bbc_i^{p,\star}.
\end{align*}
As a consequence, we obtain that the field $(\eta_K)_{K\in \{1, \ldots, N\}}$ defined by $\eta_K = \sum_{i=1}^n c_{i,K}^{p,\star} - 1$ is the solution to a backward TPFA Euler scheme for the heat equation, with diffusion coefficient $\kappa^{*,s}$ in the solid phase and $(\kappa^{*,g})^{-1}$ in the gaseous phase (the two phases decouple). We thus obtain that $\displaystyle \sum_{i=1}^n c_{i,K}^{p,\star} = 1$ for all $K\in \{1, \ldots, N\}$, by well-posedness of this scheme.

\medskip

Let us now prove the nonnegativity of $\bbc^{p,\star}$ (hence of $\bbc^p$). Let us reason by contradiction and assume that there exists $i\in \{1, \ldots, n\}$ and $K\in \{1, \ldots, N\}$ such that 
\[c^{p,\star}_{i,K} = \mathop{\min}_{j\in \{1, \ldots n\}} \mathop{\min}_{L\in \{1, \ldots, N\}} c_{j,L}^{p,\star} < 0.\]
The conservation law in $C_{K}^{p-1}$ reads
\[\frac{1}{\Delta t} (\widetilde{\Delta}_{K}^{p} c^{p,\star}_{i,K}- \Delta_{K}^{p-1}c_{i,K}^{p-1}) = J_{i,K-\frac{1}{2}}^{p,\star} - J_{i,K+\frac{1}{2}}^{p,\star},\]
from which it follows that
\[ J_{i,K-\frac{1}{2}}^{p,\star} - J_{i,K+\frac{1}{2}}^{p,\star} < 0.\]
Then a contradiction would follow if we could show that $J_{i,K-\frac{1}{2}}^{p,\star} \geq 0$ and $J_{i,K+\frac{1}{2}}^{p,\star} \leq 0$. By symmetry, it suffices to show that $J_{i,K+\frac{1}{2}}^{p,\star} \leq 0$ for the three different forms of discrete fluxes. \newline
If $K < K^{p-1}$, then according to \eqref{eq:bulk-flux-discr-s}, it holds 
\[ 
\begin{aligned}
    -\Delta x J_{i,K+\frac{1}{2}}^{p,\star} &= \kappa^{*,s} D_{K+\frac{1}{2}} \bbc_i^{p,\star} + \sum_{j=1}^n \ovkappa_{ij}^s \left( c_{j,K+\frac{1}{2}}^{p,\star} D_{K+\frac{1}{2}} \bbc_i^{p,\star} - c_{i,K+\frac{1}{2}}^{p,\star} D_{K+\frac{1}{2}} \bbc_j^{p,\star} \right) \\
    &= D_{K+\frac{1}{2}} \bbc_i^{p,\star} \l( \kappa^{*,s} + \sum_{j=1}^n \ovkappa_{ij}^s c_{j,K+\frac{1}{2}}^{p,\star} \r) \geq 0
\end{aligned}
\]
as $D_{K+1/2} c_i^{p,\star}$, $\kappa^{*,s}$, $\ovkappa_{ij}^s$ and $c_{j,K+1/2}^{p,\star}$ are all non-negative. \newline
\noindent If $K=K^{p-1}$, it holds that
\begin{equation*}
J_{i,K+\frac{1}{2}}^{p,\star} = - \widetilde{F}_i^{p,\star} = \frac{1}{\sqrt{\beta_i^*}} \l( c_{i,K^{p-1}}^{p,\star} \r)^\diamond -  \sqrt{\beta_i^*}\l(c_{i,(K^{p-1}+1)}^{p,\star} \r)^\diamond  = - \sqrt{\beta_i^*} \l(c_{i,(K^{p-1}+1)}^{p,\star} \r)^\diamond \leq 0.
\end{equation*}
If $K > K^{p-1} +1$, it holds according to \eqref{eq:bulk-flux-discr-g-implicit} 
\[ \Delta x\l( \kappa^{*,g} + \sum_{j=1}^n \ovkappa^{g}_{ij} c_{j,K+\frac{1}{2}}^{p,\star} \r) J_{i,K+\frac{1}{2}}^{p,\star} = -D_{K+\frac{1}{2}} \bbc_i^{p,\star} \leq 0, \]
from which the conclusion follows from the nonnegativity of $\Delta x, \kappa^{*,g}, \ovkappa^g_{ij}, c_{j,K+\frac{1}{2}}^{p,\star}$. We conclude that $\bbc^{p,\star} \geq 0$.

\medskip 

Let us now investigate conservation of matter. We have just proved that $(\bbc^{p}, X^p)$ is in fact a solution to the original scheme $(S)$. In particular, the fluxes are locally conservative, and it follows from summing the conservation laws \eqref{eq:conservation-discr} over the cells $K$ that, for any $i \in \{1,\dots,n\}$,
\[ \sum_{K=1}^N \Delta_{K}^{p,\star} c^{p,\star}_{i,K} = \sum_{K=1}^N \Delta_{K}^{p-1}c_{i,K}^{p-1}. \]
If $K^p=K^{p-1}$, the result follows immediately. Otherwise, fix $i \in \{1,\dots,n\}$, and let us prove that the quantity
$\sum_{K=1}^N \Delta_K^p c_{i,K}^p - \sum_{K=1}^n \Delta_K^{p,\star} c_{i,K}^{p,\star}$ is null. Observe first that, from the post-processing formulas (see Figure \ref{fig:displacement}), we only have to study the difference in the cells $C_{K^{p-1}}^{p-1}, C_{K^{p-1}+1}^{p-1},C_{K^{p-1}+2}^{p-1}$. Then compute, setting $K := K^{p-1}$,
\begin{align*}
    &\Delta_{K}^p c_{i,K}^p + \Delta_{K+1}^p c_{i,K+1}^p + \Delta_{K+2}^p c_{i,K+2}^p = \Delta x \, c_{i,K}^{p} + (X^{p}-x_{K+\frac{1}{2}}) c_{i,K+1}^p + (x_{K+\frac{5}{2}}-X^p) c_{i,K+2}^p \\
    &= \Delta x \, c_{i,K}^{p,\star} + (X^{p}-x_{K+\frac{1}{2}}) c_{i,K}^{p,\star} + \left[ (x_{K+\frac{3}{2}} - X^p) c_{i,K+1}^{p,\star} + \Delta x ~ c_{i,K+2}^{p,\star}   \right]  \\
    &= (X^p - x_{K-\frac{1}{2}}) c_{i,K}^{p,\star} + (x_{K+\frac{3}{2}} - X^p) c_{i,K+1}^{p,\star} + \Delta x \, c_{i,K+2}^{p,\star} \\
    &= \Delta_K^{p,\star} c_{i,K}^{p,\star} + \Delta_{K+1}^{p,\star} c_{i,K+1}^{p,\star} + \Delta_{K+2}^{p,\star} c_{i,K+2}^{p,\star},
\end{align*}
where we used formulas \eqref{eq:projection}-\eqref{eq:average} in the second equality. The result follows.

\end{proof}

In Lemma~\ref{lem:structure}, we proved that solutions to $(\widetilde{S})$ are \emph{a priori} nonnegative. We now prove that they are in fact (strictly) positive. 
\begin{lemma}[Strict positivity]
\label{lem:strict-pos}
 Let $p\in \mathbb{N}\setminus \{0\}$. Let $(\bbc^{p-1},X^{p-1})$ be such that for all $1\leq K \leq N$, $\bbc_{K}^{p-1}:= (c_{i,K}^{p-1})_{i\in \{1, \ldots, n\}}$ belongs to $\mathcal A$. Assume furthermore that $\sum_{K=1}^N \Delta_K^{p-1} c_{i,K}^{p-1} > 0$. Let $(\bbc^{p},X^{p})$ be a solution to $(\widetilde{S})$. Then it holds 
 \begin{equation}
     \label{eq::strict-positivity}
     c_{i,K}^p > 0, ~ \forall i \in \{1,\dots,n\}, ~ \forall K \in \{1,\dots,N\}.
 \end{equation}
\end{lemma}

\begin{proof}
The proof follows the lines of the proof of nonnegativity in Lemma~\ref{lem:structure}, but one can now take advantage of the mass conservation property. Let us reason by contradiction and assume that there exists $i\in \{1, \ldots, n\}$ and $K\in \{1, \ldots, N\}$ such that 
\[c^{p,\star}_{i,K} = \mathop{\min}_{j\in \{1, \ldots n\}} \mathop{\min}_{L\in \{1, \ldots, N\}} c_{j,L}^{p,\star} = 0.\]
Then, because of mass conservation (Lemma~\ref{lem:structure}), and since the initial mass is positive, $\bbc_i^{p,\star}$ cannot be uniformly null. By symmetry, we can assume without loss of generality that $K < N$ and $c^{p,\star}_{i,K+1} > 0$. As previously, we obtain from the conservation law that 
\[ J_{i,K-\frac{1}{2}}^{p,\star} - J_{i,K+\frac{1}{2}}^{p,\star} \leq 0,\]
where now the inequality is large. To obtain a contradiction, we need to prove that the quantity on the left-hand side is (strictly) positive. We already know from the proof of Lemma~\ref{lem:structure} that it is nonnegative, therefore we only need to show that $J_{i,K+\frac{1}{2}}^{p,\star} < 0$. There are again three cases depending on the formula for the flux. \newline 
If $K < K^{p-1}$, the conclusion follows from the previously established inequality 
\[ -\Delta x J_{i,K+\frac{1}{2}}^{p,\star} \geq \kappa^{*,s} D_{K+\frac{1}{2}} \bbc_i^{p,\star} \l(1 - \sum_{j=1}^n c_{j,K+\frac{1}{2}}^{p,\star} \r), \]
combined with the fact that, because we know that $c_{i,K+\frac{1}{2}}^{p,\star}=0$ and $c_{i,K+1}^{p,\star} > 0$, it holds 
\[ \sum_{j=1}^n c_{j,K+\frac{1}{2}}^{p,\star} = \sum_{j \neq i} c_{j,K+\frac{1}{2}}^{p,\star}  \leq \frac{1}{2} \sum_{j \neq i} \l( c_{j,K+1}^{p,\star} + c_{j,K}^{p,\star} \r) = 1 - c_{i,K+1}^{p,\star} < 1.  \]
If $K=K^{p-1}$, 
\begin{equation*}
J_{i,K+\frac{1}{2}}^{p,\star} = - \widetilde{F}_i^{p,\star} = \frac{1}{\sqrt{\beta_i^*}} \l( c_{i,K^{p-1}}^{p,\star} \r)^\diamond -  \sqrt{\beta_i^*} \l(c_{i,(K^{p-1}+1)}^{p,\star} \r)^\diamond   =  - \sqrt{\beta_i^*} c_{i,(K^{p-1}+1)}^{p,\star} < 0.
\end{equation*}
If $K > K^{p-1} + 1$, then 
\[\l( \kappa^{*,g} + \sum_{j=1}^n \ovkappa^{g}_{ij} c_{j,K+\frac{1}{2}}^{p,\star} \r) J_{i,K+\frac{1}{2}}^{p,\star} < 0, \]
and conclusion follows from the positivity of the expression between the brackets. 
\end{proof}

\subsection{Discrete free energy dissipation inequality}

Let us introduce the notation, for any $K \in \{1,\dots,N\}$,
\begin{equation}
    \label{eq:phase}
    \alpha_K^p =
    \begin{cases}
        s, \text{ if } K \leq K^p, \\
        g, \text{ if } K > K^p,
    \end{cases}
\end{equation}
so that the discrete version of the free energy functional \eqref{eq:free-energy} reads
\begin{equation}
    \label{eq:discrete-free-energy}
    \begin{aligned}
    \cH^p \l(\bbc^p,X^p \r) = \sum_{K=1}^N \Delta_K^p h_{\alpha_K^p}(\bbc_{K}^p) =  \sum_{K \leq K^p} \Delta_K^p h_s(\bbc_{K}^p) + \sum_{K > K^p} \Delta_K^p h_g(\bbc_{K}^p).
    \end{aligned}
\end{equation}
Note that, besides the explicit dependence on $(\bbc^p,X^p)$, the functional depends implicitly on $p$ through the interface cell $K^p$ (resp. through $\alpha_K^p$). We eliminate this dependence by introducing the interpolation operator $\cI_{\Delta^p}$ that maps $\bbc^p$ into the (vector-valued) piecewise constant function, defined in $(0,1)$, that interpolates the values $\bbc^p$ on the mesh defined by $(\Delta_K^p)_{K \in \{1,\dots,N\}}$. We can now connect the discrete energy functional to its continuous counterpart \eqref{eq:free-energy} as
\begin{equation}
    \cH^p\l(\bbc^p,X^p\r) = \cH \l(\cI_{\Delta^p}\l(\bbc^p\r),X^p \r).
\end{equation}
Following the modifications of the diffusion matrices \eqref{eq:As-modif}-\eqref{eq:Ag-modif}, we define the modified mobility matrices in the spirit of \eqref{eq:mobility} as, for any $\bbc \in \cA \cap (\R_+^*)^n$,
\begin{equation} 
   \label{eq:mobility-discr}
   \begin{aligned}
       \widehat{\bbM}_s(\bbc) &= \widehat{\bbA}_s(\bbc) \bbH^{-1}(\bbc), \\
       \widehat{\bbM}_g(\bbc) &= \widehat{\bbA}_g(\bbc) \bbH^{-1}(\bbc).
   \end{aligned}
\end{equation}
The positivity result of Lemma~\ref{lem:strict-pos} implies that the chain rule is valid for any $p \geq 1$ and therefore the fluxes \eqref{eq:bulk-flux-discr-s}-\eqref{eq:bulk-flux-discr-g} can be rewritten in mobility form as 
\begin{equation}
\label{eq:bulk-flux-discr-mob}
   \begin{aligned}
   \Delta x \bbJ_{K+\frac{1}{2}}^{p,\star} &= - \widehat{\bbM}_s(\bbc_{K+\frac{1}{2}}^{p,\star}) D_{K+\frac{1}{2}} \log(\bbc^{p,\star}), \; \forall 1 \leq K < K^{p-1},  \\
   \Delta x \bbJ_{K+\frac{1}{2}}^{p,\star} &= -\widehat{\bbM}_g(\bbc_{K+\frac{1}{2}}^{p,\star}) D_{K+\frac{1}{2}} \log(\bbc^{p,\star}), \; \forall K^{p-1} < K \leq N-1.
   \end{aligned}
\end{equation}
We are ready to prove a discrete version of the free energy dissipation relation \eqref{eq:dissip}, as stated in the next lemma. 
\begin{lemma}
\label{lem:dissip}
Let $(\bbc^{p-1},X^{p-1})$ be such that $\bbc^{p-1} \geq 0$ and $\sum_{i=1}^n c_{i,K}^{p-1} = 1$ for any $K \in \{1,\dots,N\}$. Let $(\bbc^p,X^p)$ be a solution to $(\widetilde{S})$. It holds
\begin{equation}
    \label{eq:discreted-dissip}
\begin{aligned}
    &\cH^p(\bbc^p,X^p) +\frac{\Delta t}{\Delta x} \sum_{K \neq K^{p-1}} (D_{K+\frac{1}{2}} \log(\bbc^{p,\star}))^T \bbM_{\alpha_K^{p-1}}(c_{K+\frac{1}{2}}^{p,\star}) D_{K+\frac{1}{2}} \log(\bbc^{p,\star}) \\
    &+\Delta t \sum_{i=1}^n F_i^{p,\star} D_{K^{p-1}+\frac{1}{2}} \l( \log(\bbc_i^{p,\star}) + \mu_i^{*} \r) \leq \cH^{p-1}(\bbc^{p-1},X^{p-1})
\end{aligned}
\end{equation}
In particular, $\cH^p(\bbc^p,X^p) \leq \cH^{p-1}(\bbc^{p-1},X^{p-1}) $.
\end{lemma}
\begin{proof}
In the same spirit as in the proof of matter conservation, we first introduce the intermediate energy quantity 
\begin{equation*}
\begin{aligned}
    \cH^{p-1}(\bbc^{p,\star},X^p) = \sum_{K=1}^N \Delta_K^p h_{\alpha_K^{p-1}}(\bbc_{K}^{p,\star}) = \sum_{K \leq K^{p-1}} \Delta_K^{p,\star} h_s(\bbc_K^{p,\star}) + \sum_{K > K^{p-1}} \Delta_K^{p,\star} h_g(\bbc_K^{p,\star}).
\end{aligned}
\end{equation*}
Using the expression of the entropy density \eqref{eq:energy-density} and conservation of matter, it holds
\begin{equation*}
\begin{aligned}
    &\cH^{p-1}(\bbc^{p,\star},X^p) - \cH^{p-1}(\bbc^{p-1},X^{p-1}) = \\
    &\sum_{K=1}^N \sum_{i=1}^n \left( \Delta_K^{p,\star} c_{i,K}^{p,\star} \left( \log(c_{i,K}^{p,\star}) + \mu_i^{*,\alpha_K^{p-1}} \right) - \Delta_K^{p-1} c_{i,K}^{p-1} \left( \log(c_{i,K}^{p-1}) + \mu_i^{*,\alpha_K^{p-1}} \right) \right).
\end{aligned}
\end{equation*}
On the other hand, multiplying the conservation laws \eqref{eq:conservation-discr} by $\Delta t \left( \log(c_{i,K}^{p,\star}) + \mu_i^{*,\alpha_K^{p-1}} \right)$ --bear in mind that $c_{i,K}^{p,\star}>0$ owing to Lemma~\ref{lem:strict-pos}--, we obtain 
\begin{equation}
\label{eq1}
\begin{aligned}
    &\sum_{i=1}^n \sum_{K=1}^N \left( \Delta_K^{p,\star} c_{i,K}^{p,\star} - \Delta_K^{p-1} c_{i,K}^{p-1} \right) \left( \log(c_{i,K}^{p,\star}) + \mu_i^{*,\alpha_K^{p-1}} \right) \\
    &= \Delta t \sum_{i=1}^n \sum_{K=1}^N \left(J_{i,K-\frac{1}{2}}^{p,\star} - J_{i,K+\frac{1}{2}}^{p,\star} \right) \left( \log(c_{i,K}^{p,\star}) + \mu_i^{*,\alpha_K^{p-1}} \right).
\end{aligned}
\end{equation}
Using the mobility form of the bulk fluxes \eqref{eq:bulk-flux-discr-mob} and applying discrete integration by parts, the right-hand side of \eqref{eq1} can be reformulated as 
\begin{align*}
&\Delta t \sum_{i=1}^n \sum_{K=1}^N \left(J_{i,K-\frac{1}{2}}^{p,\star} - J_{i,K+\frac{1}{2}}^{p,\star} \right) \left( \log(c_{i,K}^{p,\star}) + \mu_i^{*,\alpha_K^{p-1}} \right) \\
&= -\frac{\Delta t}{\Delta x} \sum_{K \neq K^{p-1}} (D_{K+\frac{1}{2}} \log(\bbc^{p,\star}))^T \bbM_{\alpha_K^{p-1}}(c_{K+\frac{1}{2}}^{p,\star}) D_{K+\frac{1}{2}} \log(\bbc^{p,\star})\\
&-\Delta t \sum_{i=1}^n F_i^{p,\star} D_{K^{p-1}+\frac{1}{2}} \l[ \log(\bbc_i^{p,\star}) + \mu_i^{*} \r]. 
\end{align*}
On the other hand, the convexity of the functional $c \to c \log c $ implies that 
\[ 
\begin{aligned}
&\sum_{i=1}^n \sum_{K=1}^N \left( \Delta_K^{p,\star} c_{i,K}^{p,\star} - \Delta_K^{p-1} c_{i,K}^{p-1} \right) \left( \log(c_{i,K}^{p,\star}) + \mu_i^{*,\alpha_K^{p-1}} \right) \\
&\geq \sum_{K=1}^N \sum_{i=1}^n \left( \Delta_K^{p,\star} c_{i,K}^{p,\star} \left( \log(c_{i,K}^{p,\star}) + \mu_i^{*,\alpha_K^{p-1}} \right) - \Delta_K^{p-1} c_{i,K}^{p-1} \left( \log(c_{i,K}^{p-1}) + \mu_i^{*,\alpha_K^{p-1}} \right) \right) \\
&= \cH^{p-1}(\bbc^{p,\star},X^p) - \cH^{p-1}(\bbc^{p-1},X^{p-1}), 
\end{aligned}\]
so inserting the two previous equations in \eqref{eq1} gives
\begin{equation}
    \label{eq2}
    \begin{aligned}
&\cH^{p-1}(\bbc^{p,\star},X^p) +\frac{\Delta t}{\Delta x} \sum_{K \neq K^{p-1}} (D_{K+\frac{1}{2}} \log(\bbc^{p,\star}))^T \bbM_{\alpha_K^{p-1}}(c_{K+\frac{1}{2}}^{p,\star}) D_{K+\frac{1}{2}} \log(\bbc^{p,\star}) \\
& +\Delta t \sum_{i=1}^n F_i^{p,\star} D_{K^{p-1}+\frac{1}{2}} \l[ \log(\bbc_i^{p,\star}) + \mu_i^{*} \r] \leq \cH^{p-1}(\bbc^{p-1},X^{p-1}) 
\end{aligned}
\end{equation}
It remains to prove the inequality
\[ \cH^p(\bbc^{p},X^p) \leq \cH^{p-1}(\bbc^{p,\star},X^p), \]
or equivalently
\[ \cH \l(\cI_{\Delta_K^p}\l(\bbc^{p}\r),X^p \r) \leq \cH \l(\cI_{\Delta_K^{p,\star}}\l(\bbc^{p,\star}\r),X^p \r). \]
The latter stems from the convexity of $\cH$ with respect to its first argument and the fact that $\cI_{\Delta_K^p}\l(\bbc^{p}\r)$ is obtained from $\cI_{\Delta_K^{p,\star}}\l(\bbc^{p,\star}\r)$ by projection \eqref{eq:projection} and convex combination \eqref{eq:average}. The proof is complete.
\end{proof}

\subsection{Existence of a discrete solution}

Thanks to the \emph{a priori} estimates that were established in the previous sections, we are now in position to prove the existence of at least one discrete solution to the scheme (S) that satisfies all the required properties. 
\begin{theorem}
\label{thm:existence}
Let $(\bbc^{p-1},X^{p-1})$ be such that $\bbc^{p-1} \geq 0$ and $\sum_{i=1}^n c_{i,K}^{p-1} = 1$ for any $K \in \{1,\dots,N\}$, $X^{p-1} \in [0,1]$. Assume in addition that $\sum_{K=1}^N \Delta_K^{p-1} c_{i,K}^{p-1} > 0$. There exists a solution $(\bbc^p,X^p)$ to $(S)$ that satisfies the properties listed in Lemmas~\ref{lem:structure}, ~\ref{lem:strict-pos} and \ref{lem:dissip}.
\end{theorem}

\begin{proof}
The proof uses the topological degree theory and in particular the properties of the degree listed in \cite[Theorem 3.1]{deimling1985}. The idea is to continuously deform our coupled system to two independent \emph{linear} systems for which we know that a solution exists, while ensuring that some \emph{a priori} estimates remain valid along the path. In fact, only the nonnegativity and volume-filling estimates are needed, since they provide boundedness in $l^\infty$ norm. We detail the argument below.

Let $\lambda\in [0,1]$. The scheme $(\widetilde{S}_\lambda)$ is constructed by introducing the following modifications to $(\widetilde{S})$. First, we introduce for any $\bbu \in \R^n$ the matrices
 \begin{equation}
     \label{eq:A-lambda}
     \begin{aligned}
         \widehat{\bbA}_s^\lambda(\bbu) &:= \kappa^{*,s} \bbI + \lambda \bar{\bbA_s}(\bbu) , \\
         \check{\widetilde{\bbA}}_g^\lambda(\bbu) &:= \kappa^{*,g} \bbI + \lambda \bar{\bbA_g}(\bbu).
     \end{aligned}
\end{equation}
Note that, adapting the proof of Lemma~\ref{lem:inversion}, $\check{\widetilde{\bbA}}_g^\lambda(\bbu)$ has positive eigenvalues uniformly bounded away from 0 and is therefore in particular invertible as soon as $\bbu \in \cA$. Then the bulk fluxes are defined as in \eqref{eq:bulk-flux-discr-s}-\eqref{eq:bulk-flux-discr-g-implicit}, using $\widehat{\bbA}_s^\lambda(\bbu)$ (resp. $\check{\widetilde{\bbA}}_g^\lambda(\bbu)$) instead of $\widehat{\bbA}_s(\bbu)$ (resp. $\check{\widetilde{\bbA}}_g(\bbu)$).  

Second, the interface fluxes \eqref{eq:interfaceflux-mod} are modified into, for any $\bbc \in (\R^n)^N$,
\begin{equation}
    \label{eq:interfaceflux-mod-lambda}
    \begin{aligned}
        \bbJ_{K^{p-1}+\frac{1}{2}}^{p, s, \lambda}(\bbc) &= - \lambda \l(\sum_{i=1}^n c_{i,K^{p-1}} \r) \widetilde{\bbF}^p(\bbc), \\
        \bbJ_{K^{p-1}+\frac{1}{2}}^{p, g,\lambda}(\bbc) &= - \lambda \l(\sum_{i=1}^n c_{i,K^{p-1}+1} \r) \widetilde{\bbF}^p(\bbc).
    \end{aligned}
\end{equation}
Finally, the interface evolves according to
\begin{equation}
\label{eq:interface-discrete-lambda}
\widetilde{X}^{p,\lambda} = X^{p-1} + \lambda \Delta t \sum_{i=1}^n \widetilde{F}_i^{p,\star}.
\end{equation}
A solution to this scheme is denoted by $(\bbc^\lambda, \widetilde{X}^{p,\lambda})$. For any $\bbc\in (\R^n)^N$, $\lambda \in [0,1]$ and $X\in [0,1]$, we denote by $h(\lambda,\bbc, X) \in (\R^n)^N \times \R$ the associated residual vector. For $\lambda=0$, the scheme boils down to two independent linear diagonal systems defined in a fixed boundary domain with zero-flux boundary conditions, for which we know that a solution exists in $\cA^N$. For $\lambda=1$, we get the scheme $(\widetilde{S})$, for which we have already proven nonnegativity and volume-filling constraint in Lemma~\ref{lem:structure}. The proof can be directly adapted to the case $\lambda \in (0,1)$, so that any solution $\bbc^\lambda$ to the scheme $(\widetilde{S}_\lambda)$ belongs to $\cA^N$. Let us now define, for $\eta > 0$, the open sets
\[ \cA_\eta^N := \left\{ \bbu \in (\R^n)^N, \inf_{\bbv \in \cA^N} \|\bbu - \bbv \|_{l^\infty} \leq \eta. \right\}  \quad \mbox{ and } \quad I_\eta = (-\eta, 1+ \eta)\]
For $\eta > 0$ sufficiently small, arguing again by continuity of the spectrum and using the uniform lower bound on the eigenvalues as in the proof of Lemma~\ref{lem:inversion}, the residual $[0,1] \times \cA_\eta^N \times I_\eta \ni (\lambda,\bbc, X) \to h(\lambda,\bbc, X)$ is a well-defined, continuous function. Moreover, the estimates give that any solution $(\bbc^\lambda, \widetilde{X}^{p,\lambda})$ to $h(\lambda,\bbc^\lambda,\widetilde{X}^{p,\lambda})=0$ lies in $\cA^N \times [0,1]$  and therefore in the interior of $\cA_\eta^N \times I_\eta$. These two ingredients allow to conclude that the topological degree of $\l(h(\lambda, \cdot),\cA_\eta^N\times I_\eta,0\r)$ is constant with respect to $\lambda$ and therefore equal to 1 when $\lambda=1$, which yields existence of a solution to $(\widetilde{S})$. This solution then satisfies all the previously established \emph{a priori} estimates and is in consequence a solution to $(S)$.

\end{proof}

\section{Numerical results}
\label{sec:numerics}

The numerical scheme has been implemented in the Julia language. The nonlinear system is solved with Newton's method, with stopping criterion $\|\bbc^{p,k+1}-\bbc^{p,k} \|_{\infty} \leq 10^{-12}$ and adaptive time stepping based on the CFL condition \eqref{eq:CFL}. We fix an initial interface $X^0=0.51$ and consider smooth initial concentrations \[c^0_{1}(x) = c^0_{2}(x) = \frac{1}{4} \l(1 + \cos(\pi x)\r), ~  c_{3}^0(x) = \frac{1}{2}\l(1-\cos(\pi x)\r)\] that are suitably discretized on a uniform mesh of $N=100$ cells. The cross-diffusion coefficients are taken equal in each phase, with values $\kappa_{12}=\kappa_{21}= 0.2, ~ \kappa_{23}=\kappa_{32}=0.1, ~ \kappa_{13}=\kappa_{31}=1$ and the numerical diffusion parameters are  $\kappa^{*,s} = \kappa^{*,g} = 0.1$ (but remember that the cross-diffusion matrices of each phase are morally inverses of each other). The solid reference chemical potential is chosen such that $\bbmu^{*,s} =0$, so that the interface dynamics only depends on $\boldsymbol{\beta^*} = \exp(\bbmu^{*,g})$. We always consider the time horizon $T=5$.

\medskip

Our first test case is devoted to the trivial situation $\boldsymbol{\beta^*} \equiv 1$. We start from a time step $\Delta t_1 = 8 \times 10^{-4}$. Snapshots of the simulation are presented in Figure~\ref{fig:trivial}, where we verify that, although discontinuities appear across the interface, the interface itself does not move, and the system converges to constant concentrations in the entire domain. Exponential decay of the relative free energy $\cH(\bbc^p,X^p)-\cH(\bbc^{\infty},X^{\infty})$ is shown in Figure~\ref{fig:trivial-longtime}. Note that, in this case, the phases are only distinguished by different speeds of convergence to equilibrium.

\begin{figure*}
  \centering
  \begin{subfigure}[b]{0.35\textwidth}
      \centering
      \includegraphics[width=\textwidth]{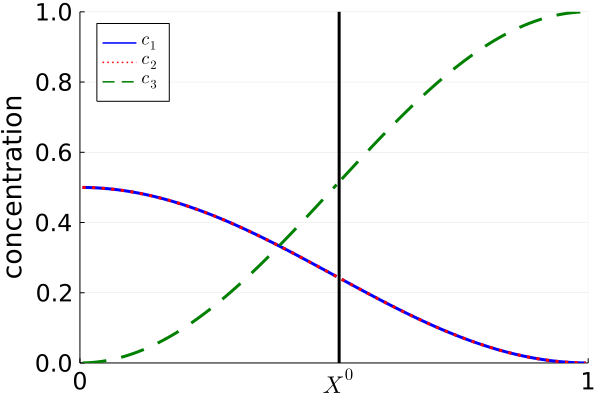}
      \caption[Network2]%
      {{\small Initial profiles}}    
  \end{subfigure}
  \quad
  \begin{subfigure}[b]{0.35\textwidth}  
      \centering 
      \includegraphics[width=\textwidth]{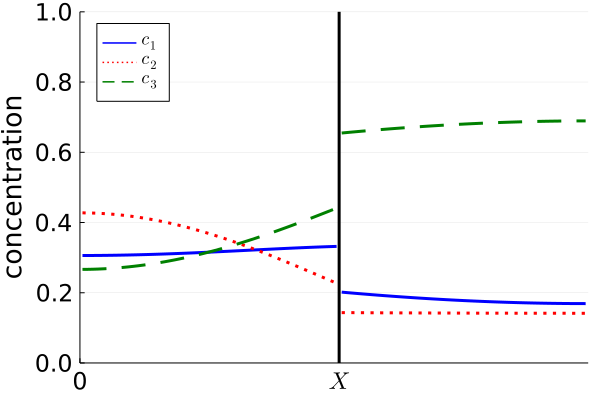}
      \caption[]%
      {{\small $t=0.25$}}    
  \end{subfigure}

  \begin{subfigure}[b]{0.35\textwidth}   
      \centering 
      \includegraphics[width=\textwidth]{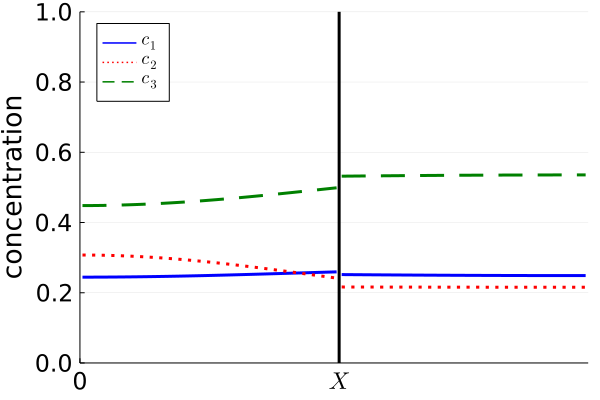}
      \caption[]
      {{\small $t=1.0$}}    
  \end{subfigure}
  \quad
  \begin{subfigure}[b]{0.35\textwidth}   
      \centering 
      \includegraphics[width=\textwidth]{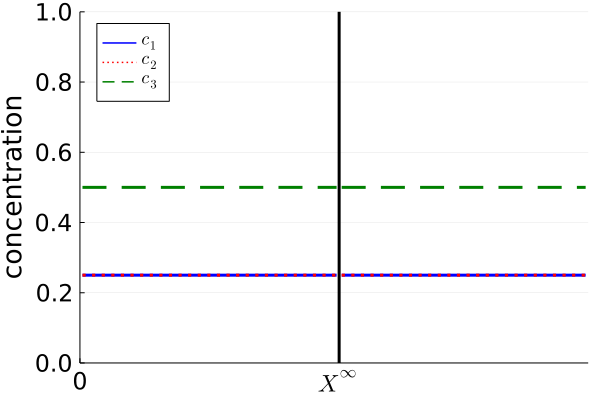}
      \caption[]%
      {{\small Stationary profiles}}  
  \end{subfigure}
  \caption[]
  {\small Trivial case: no interface movement} 
  \label{fig:trivial}
\end{figure*}

\medskip

In our second test case, we choose $\bbbeta^* = [\frac{1}{6},4,4]$, so as to fulfill the equilibrium condition \eqref{eq:two-phase-condition}, and an initial time step $\Delta t_2 = 6 \times 10^{-4}$. The simulation is presented in Figure~\ref{fig:eq-monotone}, where we observe the interface evolution and convergence in the long-time limit to the two-phase stationary solution defined by Proposition~\ref{prop:stationary}. To study the long-time asymptotics, we first compute accurately the stationary solution $(\bbc^{\infty},X^{\infty})$ (we construct the function $\varphi$ defined in \eqref{eq:varphi} and solve $\varphi(X^\infty)=0$ with Newton's method). Then, in addition to the relative free energy, we study the relative interface $|X^{\infty}-X^p|$ over time, see Figure~\ref{fig:eq-monotone-longtime}, where we observe exponential convergence and decrease of both functionals. In particular, our scheme is well-balanced and preserves the asymptotics of the continuous system. 

\begin{figure*}
  \centering
  \begin{subfigure}[b]{0.35\textwidth}
      \centering
      \includegraphics[width=\textwidth]{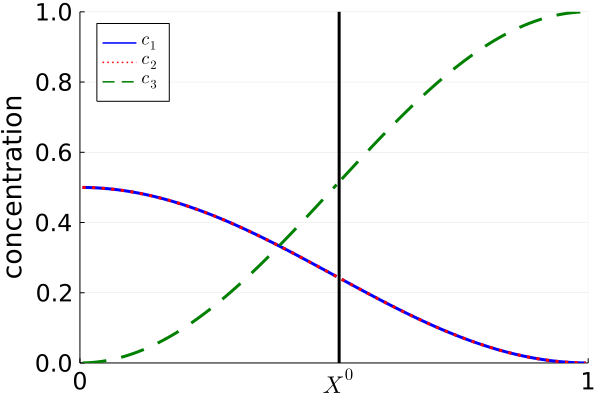}
      \caption[Network2]%
      {{\small Initial profiles}}    
  \end{subfigure}
  \quad
  \begin{subfigure}[b]{0.35\textwidth}  
      \centering 
      \includegraphics[width=\textwidth]{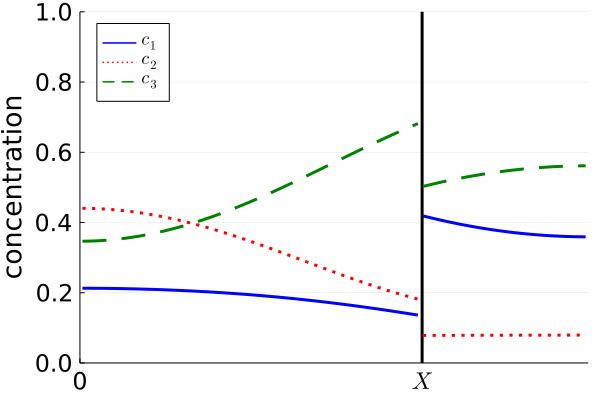}
      \caption[]%
      {{\small $t=0.25$}}    
  \end{subfigure}
  
  \begin{subfigure}[b]{0.35\textwidth}   
      \centering 
      \includegraphics[width=\textwidth]{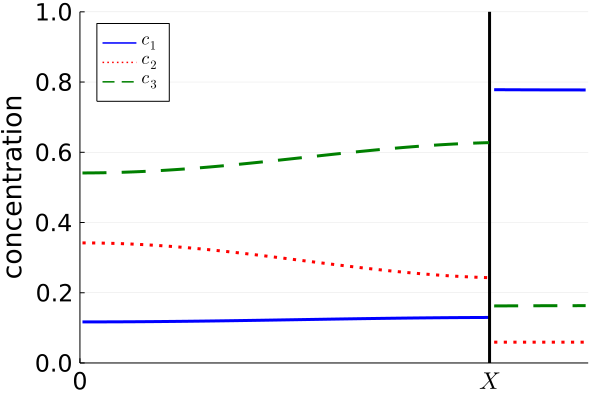}
      \caption[]
      {{\small $t=1.0$}}    
  \end{subfigure}
  \quad
  \begin{subfigure}[b]{0.35\textwidth}   
      \centering 
      \includegraphics[width=\textwidth]{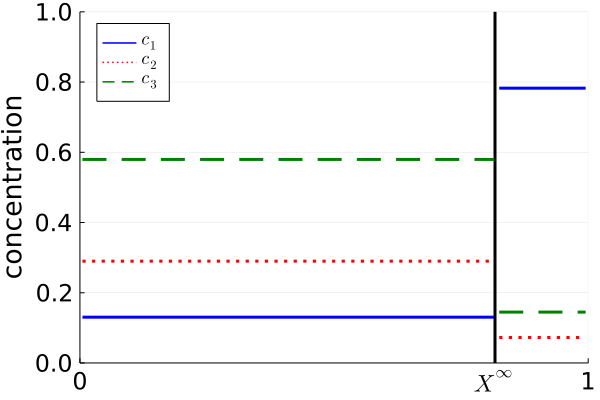}
      \caption[]%
      {{\small Stationary profiles}}  
  \end{subfigure}
  \caption[]
  {\small Equilibrium case with monotone interface } 
  \label{fig:eq-monotone}
\end{figure*}

Note that, in the previous case, the interface evolves monotonously and $|X^{\infty}-X^p| = X^\infty-X^p$. However, modifying $\bbbeta^*$, we can easily construct a test case where the interface is not monotone along the evolution, see Figures~\ref{fig:eq-non-monotone} and \ref{fig:eq-non-monotone-longtime}.
\begin{figure*}
  \centering
  \begin{subfigure}[b]{0.35\textwidth}
      \centering
      \includegraphics[width=\textwidth]{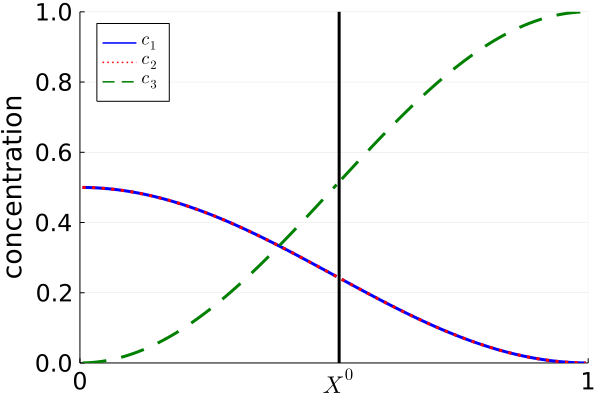}
      \caption[Network2]%
      {{\small Initial profiles}}    
  \end{subfigure}
  \quad
  \begin{subfigure}[b]{0.35\textwidth}  
      \centering 
      \includegraphics[width=\textwidth]{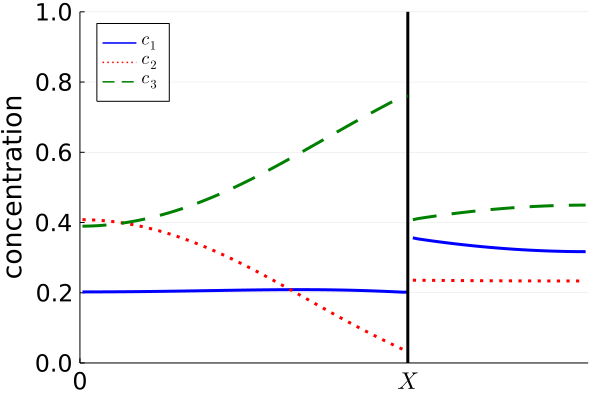}
      \caption[]%
      {{\small $t=0.25$}}    
  \end{subfigure}
  
  \begin{subfigure}[b]{0.35\textwidth}   
      \centering 
      \includegraphics[width=\textwidth]{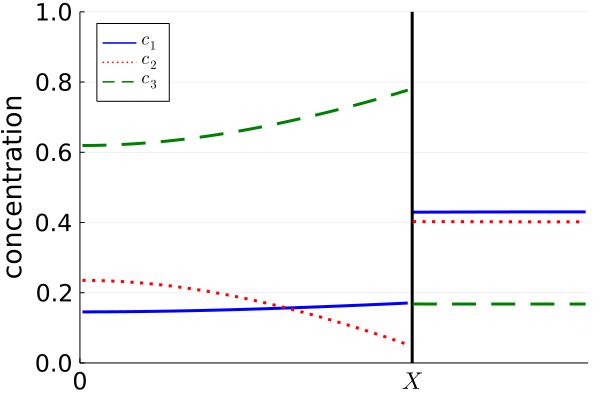}
      \caption[]
      {{\small $t=1.0$}}    
  \end{subfigure}
  \quad
  \begin{subfigure}[b]{0.35\textwidth}   
      \centering 
      \includegraphics[width=\textwidth]{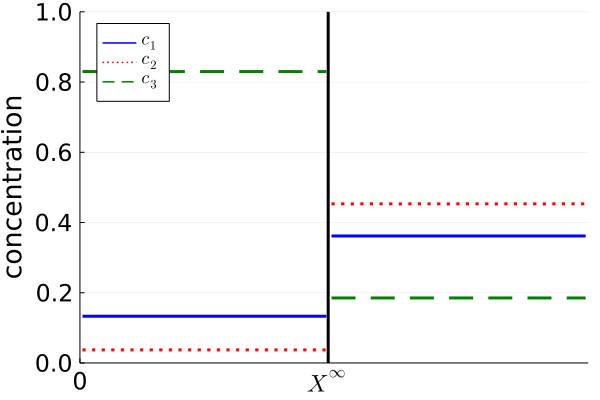}
      \caption[]%
      {{\small Stationary profiles}}  
  \end{subfigure}
  \caption[]
  {\small Equilibrium case with non-monotone interface } 
  \label{fig:eq-non-monotone}
\end{figure*}
Finally, we verify that, as soon as $\bbbeta^* \not\equiv 1$ violates \eqref{eq:two-phase-condition}, then the system converges in finite time to a one-phase solution, see Figures~\ref{fig:non-eq} and \ref{fig:non-eq-longtime}.

\medskip 

\begin{figure*}
  \centering
  \begin{subfigure}[b]{0.35\textwidth}
      \centering
      \includegraphics[width=\textwidth]{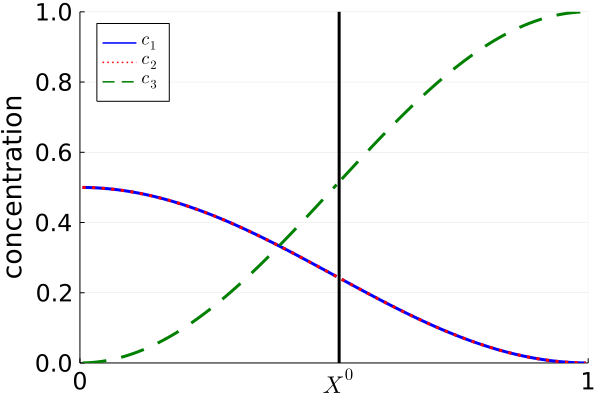}
      \caption[Network2]%
      {{\small Initial profiles}}    
  \end{subfigure}
  \quad
  \begin{subfigure}[b]{0.35\textwidth}  
      \centering 
      \includegraphics[width=\textwidth]{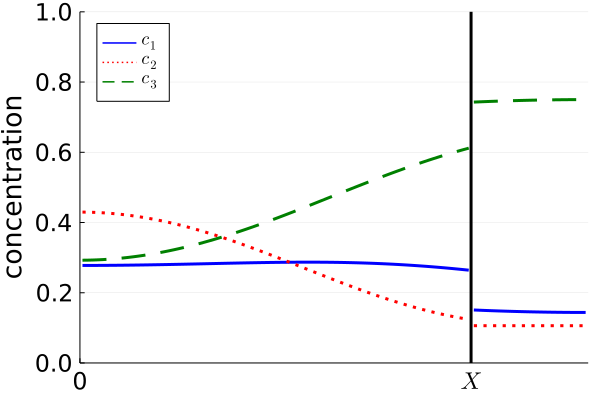}
      \caption[]%
      {{\small $t=0.25$}}    
  \end{subfigure}
  
  \begin{subfigure}[b]{0.35\textwidth}   
      \centering 
      \includegraphics[width=\textwidth]{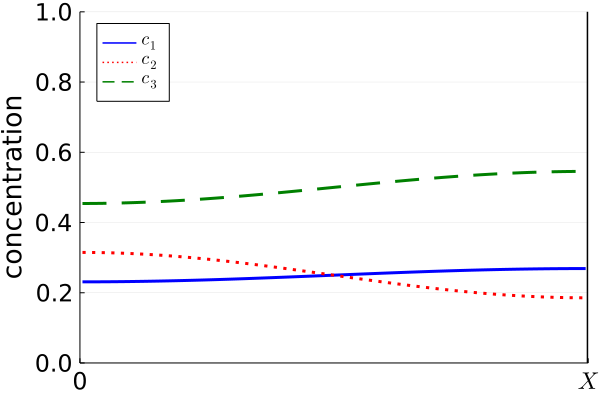}
      \caption[]
      {{\small $t=1.0$}}    
  \end{subfigure}
  \quad
  \begin{subfigure}[b]{0.35\textwidth}   
      \centering 
      \includegraphics[width=\textwidth]{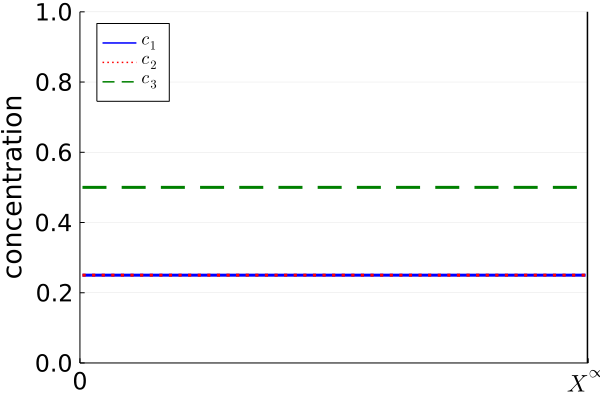}
      \caption[]%
      {{\small Stationary profiles}}  
  \end{subfigure}
  \caption[]
  {\small Non-equilibrium case } 
  \label{fig:non-eq}
\end{figure*}

\begin{figure*}
  \centering
  \begin{subfigure}[b]{0.35\textwidth}
    \centering
    \includegraphics[width=\textwidth]{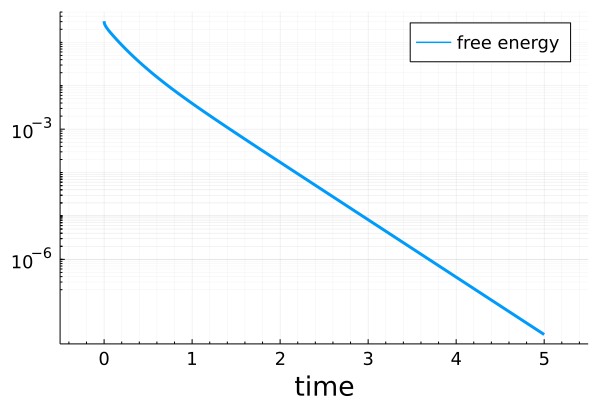}
\caption{\small Trivial case}
\label{fig:trivial-longtime}
\end{subfigure}
\quad
   \begin{subfigure}[b]{0.35\textwidth}
    \centering
    \includegraphics[width=\textwidth]{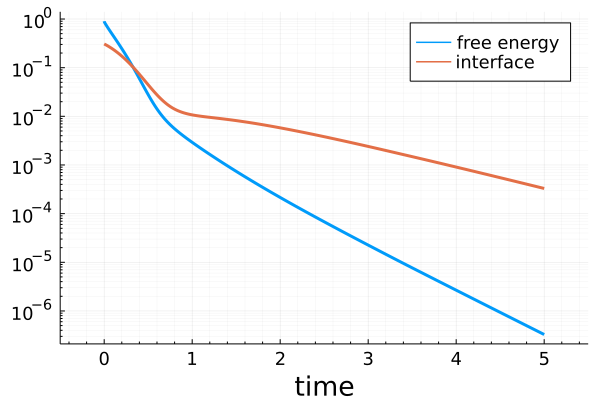}
\caption{\small Equilibrium case: monotone}
\label{fig:eq-monotone-longtime}
\end{subfigure}

  \begin{subfigure}[b]{0.35\textwidth}
    \centering
    \includegraphics[width=\textwidth]{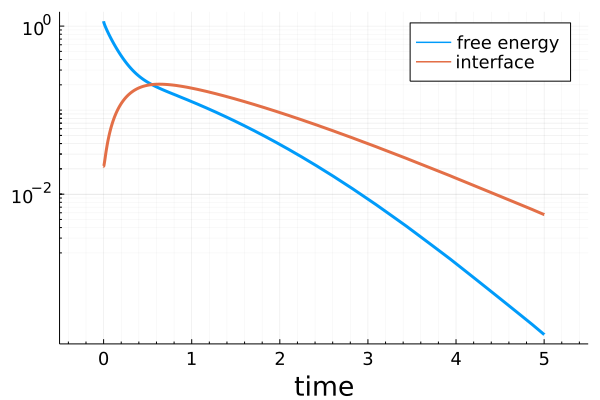}
\caption{\small Equilibrium case: not monotone}
\label{fig:eq-non-monotone-longtime}
\end{subfigure}
  \begin{subfigure}[b]{0.35\textwidth}
    \centering
    \includegraphics[width=\textwidth]{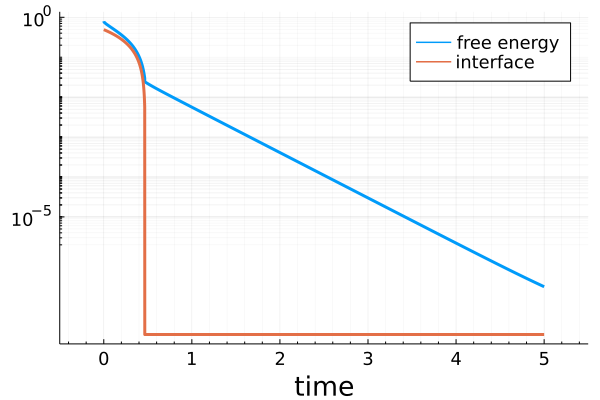}
\caption{\small Non-equilibrium case}
\label{fig:non-eq-longtime}
\end{subfigure}
  \caption{\small $\cH(\bbc^p,X^p)-\cH(\bbc^{\infty},X^{\infty})$ and $|X^{\infty}-X^p|$ for different test cases }
\end{figure*}

Our final test case is devoted to a convergence analysis with respect to the size of the mesh. We consider a fixed time step $\Delta t_2=10^{-4}$, a final time $T_2=0.25$, uniform meshes from $2^3$ to $2^{10}$ cells and we compare the different solutions with respect to a reference solution computed on a finer grid of $2^{11}$ cells. The comparison is done by projecting the certified solution onto the coarse grid using the mean value of the certified solution in the coarse cell. Then the discrete $L_{t,x}^1$ error is calculated on the coarse grid, see Figure~\ref{fig:convergence} where we also display the $L_t^1$ error on the interface. One clearly observes convergence, at first order in space for the concentrations. This indicates that the interface treatment induces a loss of order with respect to the case of a fixed interface, where the scheme is second-order accurate \cite{cancesConvergentEntropyDiminishing2020,cancesFiniteVolumesStefan2024}. 

\begin{figure*}
  \centering
      \includegraphics[width=0.5\textwidth]{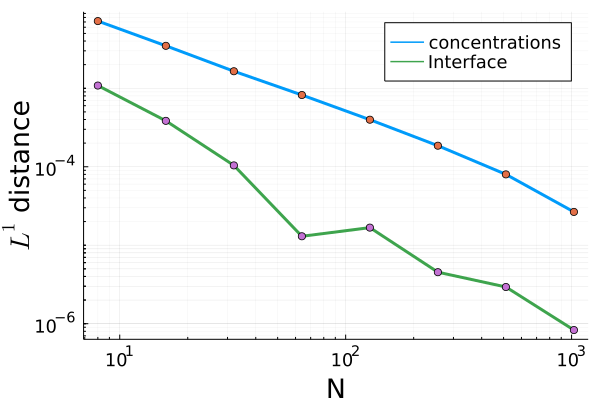}
  \caption{Convergence analysis of the solution under space grid refinement}
  \label{fig:convergence}
\end{figure*}

\section{Perspectives}
A natural perspective is to prove the convergence of the finite volume scheme presented here to some weak solution of the model~\eqref{eq:model}, which would yield in particular the existence of such a solution to the model. The study of the long-time asymptotic behaviour of such weak solutions and their discrete counterparts is also on the scientific agenda. In particular, proving the conjecture inspired by the numerical results shown in Section~\ref{sec:numerics} that the solution converges exponentially fast with respect to time to some stationary state of the model in the sense of Definition~\ref{def:stationary} is an interesting question. We intend to study these issues in a future work.

\section*{Acknowledgement}

Jean Cauvin-Vila acknowledges support from the ANR project COMODO (ANR-
19-CE46-0002) and from the European Research Council (ERC) under the European Union’s Horizon 2020 research and innovation programme, ERC Advanced Grant NEUROMORPH, no. 101018153.

\FloatBarrier
\printbibliography

\end{document}